\documentclass[11 pt]{amsart}

\usepackage{amsmath,amssymb,amsthm}
\usepackage[top=1.5in,bottom=1.3in,left=1.3in,right=1.3in,marginparwidth=1.5cm]{geometry}

\usepackage{mathtools}
\usepackage{enumitem}
\usepackage[hidelinks]{hyperref}
\usepackage[svgnames]{xcolor}
\allowdisplaybreaks

\title[Classification of   irreducible $N_k(\mathfrak{sl}_2)$--modules] {Classification of irreducible highest weight modules for the parafermion vertex algebras   $N_k(\mathfrak{sl}_2)$ at arbitrary level}

\author[Dra\v zen Adamovi\'c]{Dra{\v z}en~Adamovi\'c}
\address[Dra{\v z}en~Adamovi\'c]{Department of Mathematics, Faculty of Science, University of Zagreb, Bijeni\v cka 30, 10 000 Zagreb, Croatia}
\email{adamovic@math.hr}

\author[Qing Wang]{ Qing Wang}
\address[Qing Wang]{School of Mathematical Sciences, Xiamen University, Xiamen, China 361005}
\email{qingwang@xmu.edu.cn}

\date{}

\newtheorem{theorem}{Theorem}
\newtheorem{lemma}{Lemma}
\newtheorem{proposition}{Proposition}
\newtheorem{corollary}{Corollary}
\newtheorem{remark}{Remark}
\theoremstyle{definition}
\newtheorem{definition}{Definition}
\theoremstyle{plain}

\newcommand{\Vir}{\mathrm{Vir}}
\newcommand{\one}{\mathbf 1}

\begin{document}

\begin{abstract}
Let $N^k(\mathfrak{sl}_2)$ be the universal parafermion vertex algebra
and $N_k(\mathfrak{sl}_2)$ its simple quotient.
We classify all irreducible highest weight
$N^k(\mathfrak{sl}_2)$--modules for every $k\neq0$.
We give a presentation of Zhu's algebra
$A(N^k(\mathfrak{sl}_2))$ as a quotient of a polynomial algebra in
four variables by an ideal generated by three explicit polynomials.
This presentation also shows that $A(N^k(\mathfrak{sl}_2))$ is a free
module of rank three over the polynomial subalgebra generated by the
classes of the fields of weights two and three.
The irreducible highest weight $N^k(\mathfrak{sl}_2)$--modules are
parametrized by a two-parameter family $L_k[x,y]$,
$(x,y)\in\mathbb C^2$, constructed using a free-field realization.

We also prove that each $N^k(\mathfrak{sl}_2)$--module $L_k[x,y]$
can be realized as an $N^k(\mathfrak{sl}_2)$--submodule of an
irreducible weight module $M$ for the universal affine vertex algebra
$V^k(\mathfrak{sl}_2)$.
At non-integral  admissible levels, we prove that $L_k[x,y]$ is an
$N_k(\mathfrak{sl}_2)$--module if and only if the associated
$V^k(\mathfrak{sl}_2)$--module $M$ is an
$L_k(\mathfrak{sl}_2)$--module.
This gives the classification of irreducible highest weight
$N_k(\mathfrak{sl}_2)$--modules at all non-integral admissible levels.
We also classify the irreducible highest weight modules for the simple
parafermion algebra $N_{-2}(\mathfrak{sl}_2)$ at the critical level.
At positive integer levels, the irreducible
$N_k(\mathfrak{sl}_2)$--modules were classified in \cite{ALY}.
\end{abstract}

\maketitle

\section{Introduction}
Let $V^k(\mathfrak{sl}_2)$ be the universal affine vertex algebra and
let $\mathcal H^k\subset V^k(\mathfrak{sl}_2)$ be the Heisenberg
vertex subalgebra generated by $h(-1)\mathbf 1$. For $k\neq0$, we set
\[
  N^k(\mathfrak{sl}_2)
  =\operatorname{Com}(\mathcal H^k,V^k(\mathfrak{sl}_2)).
\]
Let $N_k(\mathfrak{sl}_2)$ be its simple quotient.

Recall that a level $k$ is called admissible if
$k+2=\frac{u}{v}$, where $u$ and $v$ are coprime positive integers
and $u\geq2$. The level $k=-2$ is called the critical level. It is
well known that $V^k(\mathfrak{sl}_2)$ is non-simple if and only if
$k$ is admissible or critical  \cite{GK}.   Therefore
$
N^k(\mathfrak{sl}_2)\neq N_k(\mathfrak{sl}_2)
$
if and only if $k$ is admissible or $k$ is critical.

We first present the  result on strong generators of $N^k(\mathfrak{sl}_2)$.

\begin{theorem}[cf.\ \cite{DLWY}, Theorem~\ref{thm:strong-generation}]
\label{thm:strong-generation-uvod}
For every $k\neq0$, the vertex algebra $N^k(\mathfrak{sl}_2)$ is strongly
generated by fields of weights $2,3,4,5$.
\end{theorem}

This statement first appeared in the physics literature \cite{BEHHH}.
It was proved in \cite{DLWY} for positive integer levels. We give a proof
for all levels, including the critical level, in
Theorem~\ref{thm:strong-generation}. For
$k\notin\{-2,-17/16,-107/64\}$, we use the strong generators
$L,W=W^3,W^4,W^5$ from \cite{BEHHH,DLY}. At the two exceptional
levels we use instead $Z^4=W_{(1)}W$ and $Z^5=W_{(1)}Z^4$.
At $k=-2$, the field $L$ is replaced by the Segal--Sugawara field
$S$. The parafermion vertex algebra $N^{-3/2}(\mathfrak{sl}_2)$ was studied in \cite{AMW}.

Our first goal is to classify irreducible highest weight
$N^k(\mathfrak{sl}_2)$-modules using Zhu's algebra
$A(N^k(\mathfrak{sl}_2))$. The top space of
an irreducible highest weight module is one-dimensional. The action of
$A(N^k(\mathfrak{sl}_2))$ on this space is therefore determined by a
character.

\subsection*{Construction of highest weight $N^k(\mathfrak{sl}_2)$--modules}
In Section~\ref{sec:vir-heis-realization}, by applying the realization $V^k(\mathfrak{sl}_2)\hookrightarrow
V^{\Vir}(c_k,0)\otimes  \Pi(0)$
from \cite{A2019}, we obtain
\[
N^k(\mathfrak{sl}_2)\hookrightarrow
V^{\Vir}(c_k,0)\otimes\mathcal H,
\]
where $V^{\Vir}(c_k,0)$ is the universal Virasoro vertex algebra and
$\mathcal H$ is the rank-one Heisenberg vertex algebra.  For each
$(x,y)\in\mathbb C^2$, we consider the module
$L^{\Vir}(c_k,y)\otimes\mathcal F_x$.
Here $L^{\Vir}(c_k,y)$ is an irreducible highest weight Virasoro module,
and $\mathcal F_x$ is an irreducible Heisenberg Fock module.  The module
$L_k[x,y]$ is the irreducible quotient of the cyclic
$N^k(\mathfrak{sl}_2)$--submodule generated by the tensor product of
the two highest weight vectors.

In Section~\ref{sec:evaluation}, we calculate the action of the strong
generators on the top space of $L_k[x,y]$.  Let
$\chi^k_{x,y}:A(N^k(\mathfrak{sl}_2))\to\mathbb C$ be the corresponding
character. For
$k\neq-2,-17/16,-107/64$, its values on the generators are
\[
\bigl(\chi^k_{x,y}([L]),\chi^k_{x,y}([W]),
\chi^k_{x,y}([W^4]),\chi^k_{x,y}([W^5])\bigr)
=
\bigl(f_2(k,x,y),f_3(k,x,y),f_4(k,x,y),f_5(k,x,y)\bigr).
\]
At the two exceptional levels, the corresponding values on
$[L],[W],[Z^4],[Z^5]$ are given by $g_2,g_3,g_4,g_5$.
By Zhu's correspondence \cite{FZ,Zhu}, two irreducible highest weight
modules with the same $A(N^k(\mathfrak{sl}_2))$-character are
isomorphic. The
formulas for $f_i$ and $g_i$ are given in
Section~\ref{sec:evaluation}.

At the critical level, the Virasoro factor is replaced by the
commutative vertex algebra generated by the Segal--Sugawara field $S$.
The corresponding characters are denoted by
$\chi^{-2}_{x,y}$. Their values on
$[S],[W],[W^4],[W^5]$ are given in
Section~\ref{sec:critical-level}.

\subsection*{Zhu's algebra $A(N^k(\mathfrak{sl}_2))$ and classification
of irreducible $N^k(\mathfrak{sl}_2)$--modules}

By Lemma~\ref{lem:zhu-commutative-uniform}, Zhu's algebra
$A(N^k(\mathfrak{sl}_2))$ is commutative.
In Section~\ref{sec:zhu-polynomial-system}, we determine this algebra,
and in Section~\ref{sec:zhu-parametrization}, we classify its
irreducible representations.

  The precise
presentations are given in Theorem~\ref{thm:zhu-presentation} for
$k\neq-2$ and in Theorem~\ref{thm:critical-zhu-presentation} for
$k=-2$.  We recall here the parts of these theorems that are used for
the classification.  If
$k\neq0,-2,-17/16,-107/64$, we prove that
\[
A(N^k(\mathfrak{sl}_2))
\cong
\frac{\mathbb C[w_2,w_3,w_4,w_5]}{(R_0,R_1,R_2)}.
\]
At the two exceptional levels $k=-17/16,-107/64$, we use the
generators $L,W,Z^4,Z^5$ and obtain
\[
A(N^k(\mathfrak{sl}_2))
\cong
\frac{\mathbb C[w_2,w_3,z_4,z_5]}
{(R_0^z,R_1^z,R_2^z)}.
\]
At the critical level,
\[
A(N^{-2}(\mathfrak{sl}_2))
\cong
\frac{\mathbb C[y,w_3,w_4,w_5]}
{(\mathcal R_0,\mathcal R_1,\mathcal R_2)}.
\]
The polynomials in these presentations are given in
Appendix~\ref{app-a}.

The same theorems prove that
$1,w_4,w_5$ form a basis over $\mathbb C[w_2,w_3]$ in the first
presentation, and that $1,z_4,z_5$ form such a
basis in the exceptional cases.  At the critical level,   $1,w_4,w_5$ form a basis over
$\mathbb C[y,w_3]$.

We use these presentations to classify the irreducible
$N^k(\mathfrak{sl}_2)$--modules.    We obtain:

\begin{theorem}[cf.\ Theorems~\ref{class-univ} and
\ref{thm:critical-universal-classification}]
For every $k\neq0$, every irreducible highest-weight
$N^k(\mathfrak{sl}_2)$--module is isomorphic to $L_k[x,y]$ for some
$(x,y)\in\mathbb C^2$.
\end{theorem}

The critical level $k=-2$ is treated separately in
Section~\ref{sec:critical-level}. We determine
$A(N^{-2}(\mathfrak{sl}_2))$ and classify the irreducible highest weight
modules for both $N^{-2}(\mathfrak{sl}_2)$ and
$N_{-2}(\mathfrak{sl}_2)$. The irreducible highest weight
$N_{-2}(\mathfrak{sl}_2)$-modules are the modules $L_{-2}[x,0]$
constructed from rank-one   modules for Heisenberg vertex algebra.

\subsection*{Cubic structure of Zhu's algebra
$A(N^k(\mathfrak{sl}_2))$}

The presentation theorems show that
$A(N^k(\mathfrak{sl}_2))$ is a free module of rank three over the
polynomial subalgebra generated by the classes of the fields of weights
two and three. Hence it is a cubic algebra over this polynomial
subalgebra, and the corresponding morphism of spectra is a triple cover
\cite{Miranda,Wood}. In Section~\ref{sec:cubic-zhu}, we give an explicit
presentation of this cubic algebra in Miranda's trace-zero normal form.
For a recent account of triple covers, see
\cite[Section~2]{GarbagnatiPenegini}.

\subsection*{Relation with tensor-category methods}
At non-integral admissible levels, the category of finitely generated
weight $L_k(\mathfrak{sl}_2)$-modules admits a vertex tensor category
structure \cite{CreutzigTensorSL2}. The rigidity of this category was
proved independently in \cite{CMYRibbonSL2} and \cite{NORW}.
Huang recently proved that, for any grading-restricted M\"obius vertex
algebra, the category of $C_1$-cofinite grading-restricted generalized
modules admits a vertex tensor category structure \cite{HuangC1}.
Recall that $C_1(M)$ is
spanned by the vectors $a_{(-1)}m$, where $a$ has positive conformal
weight, and that $M$ is $C_1$-cofinite if
$\dim M/C_1(M)<\infty$. Huang's result cannot be used to classify all
irreducible highest weight $N^k(\mathfrak{sl}_2)$-modules. For generic
$y$, the Virasoro module $L^{\Vir}(c_k,y)$ used in our construction is
an irreducible Verma module and is not $C_1$-cofinite. The proof of
\cite[Theorem~7.6]{CreutzigTensorSL2}, together with
\cite[Proposition~7.4]{CreutzigTensorSL2}, gives the same conclusion
for the corresponding $N^k(\mathfrak{sl}_2)$-module. Thus some of the
modules occurring in the classification are not $C_1$-cofinite.

One might also try to apply the Schur--Weyl theory for Heisenberg
cosets from \cite{CKLR}. The results of \cite{CKLR} assume the
existence of the required vertex tensor categories. This assumption
is not known for a category containing all irreducible highest weight
$N^k(\mathfrak{sl}_2)$-modules. Therefore, one cannot assume a priori that
every such module occurs in the restriction of an irreducible weight
$V^k(\mathfrak{sl}_2)$-module.

We first classify the irreducible highest weight
$N^k(\mathfrak{sl}_2)$-modules using Zhu's algebra. We then use
inverse quantum Hamiltonian reduction to realize them inside
irreducible weight $V^k(\mathfrak{sl}_2)$-modules and determine which
of them are $N_k(\mathfrak{sl}_2)$--modules. The highest weight
classification is a classical representation-theoretic problem and
does not depend on the existence of a tensor category. It also
provides a framework for future work on tensor categories of
$N_k(\mathfrak{sl}_2)$-modules at admissible levels.

\subsection*{Classification of irreducible
$N_k(\mathfrak{sl}_2)$--modules}

The irreducible $N_k(\mathfrak{sl}_2)$-modules at positive integer
levels were classified in \cite[Theorem~8.2]{ALY} (cf.\ \cite{DongRen}). Our goal is to
classify the irreducible highest weight
$N_k(\mathfrak{sl}_2)$-modules at non-integral admissible levels.
We first prove that every irreducible highest weight
$N^k(\mathfrak{sl}_2)$-module $L_k[x,y]$ is realized as
\[
\Omega_\alpha(M)=
\{m\in M\mid h(n)m=0\ (n>0),\ h(0)m=\alpha m\}
\]
for an irreducible weight $V^k(\mathfrak{sl}_2)$-module $M$ constructed
by inverse quantum Hamiltonian reduction. We then prove that
$L_k[x,y]\cong\Omega_\alpha(M)$ is an $N_k(\mathfrak{sl}_2)$-module if
and only if $M$ is an $L_k(\mathfrak{sl}_2)$-module.

At a non-integral admissible level $k=-2+\frac{u}{v}$, $u\ge 2$, $v\ge2$,
the irreducible $L_k(\mathfrak{sl}_2)$--modules with finite-dimensional
weight spaces were classified in \cite{AM95}.  In the terminology used later
in \cite{ACR}, these modules are the typical relaxed highest-weight modules
$E_{\lambda;\Delta^{\mathrm{aff}}_{r,s}}$, the highest-weight modules
$D^+_{r,s}$, their conjugates $D^-_{r,s}$, and the ordinary modules $L_r$.
The restrictions of the first three types give one-parameter families
$L_k[x,h]$, $x\in\mathbb C$, of $N_k(\mathfrak{sl}_2)$--modules,
where $h$ belongs to the Kac table
$\mathcal H_{u,v}=\{h^{u,v}_{r,s}\mid
1\le r\le u-1,\ 1\le s\le v-1\}$ of the Virasoro
$(u,v)$--minimal model.
The restrictions of ordinary modules give the following additional
$N_k(\mathfrak{sl}_2)$--modules, which do not belong to the continuous
families:
\begin{equation}
  \mathcal O^{\mathrm{red}}_{u,v}
  =
  \left\{
  L_k\left[
  \frac{k-r+1+2i}{2},
  \frac{v(r^2-1)}{4u}-\frac{k}{4}
  \right]
  \ \middle|\
  3\le r\le u-1,
  \ 1\le i\le r-2
  \right\}. \nonumber
\end{equation}

\begin{theorem}
Let $k+2 = \frac{u}{v}$ be a non-integral admissible level. The set
\[
 \{L_k[x,h]\mid x\in\mathbb C,\ h\in\mathcal H_{u,v}\}
 \cup\mathcal O^{\mathrm{red}}_{u,v}
\]
exhausts the irreducible highest weight
$N_k(\mathfrak{sl}_2)$--modules.
\end{theorem}
As another application of our classification for $N_k(\mathfrak{sl}_2)$ at levels $k=-\frac{1}{2}$ and $k=-\frac{4}{3}$, we recover the known
classifications of the singlet vertex algebras $\mathcal M(2)$ and $\mathcal M(3)$ (cf. \cite{A-2003, WangW3}).
The vertex  algebra $N_{-2/3}(\mathfrak{sl}_2)$ is isomorphic to the
$\mathbb Z_2$-orbifold of the supersinglet vertex superalgebra introduced in
\cite{AMSuperTriplet}. We prove that every irreducible highest weight
$N_{-2/3}(\mathfrak{sl}_2)$-module occurs as an irreducible submodule in
the restriction of an irreducible untwisted or twisted supersinglet
module, this verifies the orbifold conjecture for the supersinglet vertex superalgebra.

\section{Setup}
\label{sec:preliminaries}
This section fixes the notation and recalls the basic definitions used
throughout the paper.
\begin{itemize}[leftmargin=2em]
\item $V^k(\mathfrak{sl}_2)$ is the universal affine vertex algebra and
$L_k(\mathfrak{sl}_2)$ is its simple quotient.  The maximal ideal of
$V^k(\mathfrak{sl}_2)$ is denoted by
$\mathcal I_k$.
\item $\mathcal H^k\subset V^k(\mathfrak{sl}_2)$ is the Heisenberg vertex
algebra generated by $h(-1)\mathbf 1$, and
$N^k(\mathfrak{sl}_2)=
\operatorname{Com}(\mathcal H^k,V^k(\mathfrak{sl}_2))$.
The simple parafermion vertex algebra is
$N_k(\mathfrak{sl}_2)=N^k(\mathfrak{sl}_2)/
(\mathcal I_k\cap N^k(\mathfrak{sl}_2))$.
\item For $k\neq-2$, the central charge of
$N^k(\mathfrak{sl}_2)$ is
$
c_{N^k} = \frac{2(k-1)}{k+2}.
$

\item $V^{\Vir}(c,0)$ denotes the universal Virasoro vertex algebra of central
charge $c$, and $L^{\Vir}(c,0)$ denotes its simple quotient.  For
$h\in\mathbb C$, $V^{\Vir}(c,h)$ denotes the universal highest weight
Virasoro module and $L^{\Vir}(c,h)$ its irreducible quotient.
\item $\Pi(0)$ denotes the half-lattice vertex algebra, and
$\Pi\cdot e^{r\mu+x\mathbf c}$ denotes the irreducible
$\Pi(0)$-module generated by $e^{r\mu+x\mathbf c}$.
\item $L_k[x,y]$ denotes the irreducible highest weight
$N^k(\mathfrak{sl}_2)$-module with highest weight vector
$v_{x,y}$ constructed in Section~\ref{sec:evaluation} for
$k\neq-2$ and in Section~\ref{sec:critical-level} for $k=-2$.

\item The highest weight of $L_k[x,y]$ is the
$A(N^k(\mathfrak{sl}_2))$-character $\chi^k_{x,y}$. If
$k\neq-2,-17/16,-107/64$, the tuple of its values on
$[L],[W],[W^4],[W^5]$ is
$(f_2(k,x,y),f_3(k,x,y),f_4(k,x,y),f_5(k,x,y))$ and, at the
exceptional levels $k=-17/16,-107/64$, the corresponding tuple is
$(g_2(k,x,y),g_3(k,x,y),g_4(k,x,y),g_5(k,x,y))$.
At $k=-2$, its values on $[S],[W],[W^4],[W^5]$ are given by
\eqref{eq:critical-w3}--\eqref{eq:critical-w5}.
\end{itemize}

\section{Strong generators of $N^k(\mathfrak{sl}_2)$}
\label{sec:strong-generators}

For positive integer levels, the strong-generation result below was proved
in \cite{DLY,DLWY,DW1}.  The induction used in these papers extends to arbitrary
nonzero levels, except when one of its coefficients vanishes.  We give an
alternative induction step at these levels and treat the critical level
separately.

Let $V^k(\mathfrak{sl}_2)$ be the universal affine vertex algebra associated
with $\mathfrak{sl}_2$ at level $k$.  We use the standard generators
$e(z),h(z),f(z)$, normalized by
\[
\begin{aligned}
h(z)h(w)&\sim \frac{2k}{(z-w)^2},
& h(z)e(w)&\sim \frac{2e(w)}{z-w},
& h(z)f(w)&\sim -\frac{2f(w)}{z-w},\\
e(z)f(w)&\sim \frac{k}{(z-w)^2}+\frac{h(w)}{z-w}.
\end{aligned}
\]
Let $\mathcal H^k\subset V^k(\mathfrak{sl}_2)$ be the Heisenberg vertex
subalgebra generated by $h$.  For $k\neq 0$, the universal parafermion vertex
algebra is defined as the commutant
\[
N^k(\mathfrak{sl}_2)=\operatorname{Com}\bigl(\mathcal H^k,V^k(\mathfrak{sl}_2)\bigr)
=\{a\in V^k(\mathfrak{sl}_2)\mid h(n)a=0\text{ for all }n\geq 0\}.
\]
We use the standard nonnegative grading of
$V^k(\mathfrak{sl}_2)$, determined by
\[
  \deg\mathbf1=0,\qquad
  \deg\bigl(a(-n)v\bigr)=n+\deg v
  \quad
  (a\in\mathfrak{sl}_2,\ n\geq1).
\]
The commutant $N^k(\mathfrak{sl}_2)$ inherits this grading.

For $k\neq0,-2$, the conformal vector of $N^k(\mathfrak{sl}_2)$ is
$\omega_N=\omega_{\mathrm{aff}}-\omega_{\mathcal H}$, where
\[
\omega_{\mathrm{aff}}
=\frac{1}{2(k+2)}\left(\frac{1}{2}:hh:+:ef:+:fe:\right),
\qquad
\omega_{\mathcal H}=\frac{1}{4k}:hh:.
\]
We denote the corresponding Virasoro field by
$L(z)=Y(\omega_N,z)$. Its central charge is
$c_N=3k/(k+2)-1=2(k-1)/(k+2)$.

For $k\neq0,-2$, the algebra $N^k(\mathfrak{sl}_2)$ contains fields
$L$, $W=W^3$, $W^4$, and $W^5$ of conformal weights $2,3,4,5$,
respectively.  We use the normalization of
\cite{DLY}.  In the following mode formulas, the final vacuum vector
$\mathbf 1$ is omitted; for instance,
$h(-3)e(-1)f(-1)$ means $h(-3)e(-1)f(-1)\mathbf 1$.

With this convention, the conformal vector may be written as
\[
L=\omega_N
=\frac{1}{2k(k+2)}\left(
-kh(-2)-h(-1)^2+2k e(-1)f(-1)
\right).
\]
The primary field of weight three is
\[
\begin{aligned}
W^3={}& k^2 h(-3)+3k h(-2)h(-1)+2h(-1)^3
  -6k h(-1)e(-1)f(-1) \\
&\quad +3k^2 e(-2)f(-1)-3k^2 e(-1)f(-2).
\end{aligned}
\]
The primary field of weight four is
\[
\begin{aligned}
W^4={}& -2k^2(k^2+k+1)h(-4)
-8k(k^2+k+1)h(-3)h(-1)
-k(5k^2-6)h(-2)^2 \\
&\quad -2k(11k+6)h(-2)h(-1)^2
-(11k+6)h(-1)^4
+4k^2(6k-5)h(-2)e(-1)f(-1) \\
&\quad +4k(11k+6)h(-1)^2e(-1)f(-1)
-4k^2(5k+11)h(-1)e(-2)f(-1) \\
&\quad +4k^2(5k+11)h(-1)e(-1)f(-2)
+8k^2(k-3)(k-2)e(-3)f(-1) \\
&\quad -4k^2(3k^2-3k+8)e(-2)f(-2)
-2k^2(6k-5)e(-1)^2f(-1)^2 \\
&\quad +8k^2(k^2+k+1)e(-1)f(-3).
\end{aligned}
\]
The primary field of weight five is
\[
\begin{aligned}
W^5={}& -2k^3(k^2+3k+5)h(-5)
-10k^2(k^2+3k+5)h(-4)h(-1) \\
&\quad -5k^2(3k^2-4)h(-3)h(-2)
-5k(7k^2+12k+16)h(-3)h(-1)^2 \\
&\quad -15k(3k^2-4)h(-2)^2h(-1)
-5k(19k+12)h(-2)h(-1)^3
-2(19k+12)h(-1)^5 \\
&\quad +10k^2(4k^2-7k+8)h(-3)e(-1)f(-1)
+20k^2(10k-7)h(-2)h(-1)e(-1)f(-1) \\
&\quad +10k(19k+12)h(-1)^3e(-1)f(-1)
-5k^2(11k^2-14k+12)h(-2)e(-2)f(-1) \\
&\quad -5k^2(17k+64)h(-1)^2e(-2)f(-1)
+15k^2(3k^2-4)h(-2)e(-1)f(-2) \\
&\quad +5k^2(17k+64)h(-1)^2e(-1)f(-2)
+30k^2(k-4)(k-3)h(-1)e(-3)f(-1) \\
&\quad -40k^2(k^2+3k+5)h(-1)e(-2)f(-2)
-10k^2(10k-7)h(-1)e(-1)^2f(-1)^2 \\
&\quad +10k^2(3k^2+19k+8)h(-1)e(-1)f(-3)
-10k^3(k-4)(k-3)e(-4)f(-1) \\
&\quad +20k^3(k-4)(k-3)e(-3)f(-2)
+5k^3(10k-7)e(-2)e(-1)f(-1)^2 \\
&\quad -10k^3(2k^2-4k+17)e(-2)f(-3)
-5k^3(10k-7)e(-1)^2f(-2)f(-1) \\
&\quad +10k^3(k^2+3k+5)e(-1)f(-4).
\end{aligned}
\]

At the critical level $k=-2$, the field $L$ is not defined.  We use instead
the Segal--Sugawara field
\begin{equation}\label{eq:S}
S=:ef:+:fe:+\frac12:hh:.
\end{equation}
It generates the center of $V^{-2}(\mathfrak{sl}_2)$.  The fields
$W^3,W^4,W^5$ are defined at $k=-2$ by the same formulas as in
the noncritical case.  They are homogeneous of degrees $3,4,5$ in the standard
grading.

The OPE relations among these generators, including the
products $W^i_{(n)}W^j$ for $3\leq i\leq j\leq 5$, are written explicitly in
\cite[Appendices A and B]{DLY}.

We shall also use the fields $Z^4=W_{(1)}W$ and
$Z^5=W_{(1)}Z^4$.
The formulas in Section~\ref{sec:vir-heis-realization} relate
$Z^4,Z^5$ to the primary fields $W^4,W^5$. At $k=-17/16$, the
coefficient needed to recover $W^4$ vanishes, and at $k=-107/64$, the
coefficient needed to recover $W^5$ vanishes. At these two levels, we
use $Z^4,Z^5$ as strong generators. The fields $Z^4,Z^5$ are not
primary.

We obtain the following extension of the strong-generation result of
\cite{DLWY}.  The formulas for the generators and their OPE relations are
those of \cite{DLY}.

\begin{theorem}\label{thm:strong-generation}
For every $k\neq0$, the vertex algebra $N^k(\mathfrak{sl}_2)$ is strongly
generated as follows.
\begin{enumerate}[label=\textup{(\roman*)}]
\item If $k\neq0,-2,-17/16,-107/64$, then
\[
  N^k(\mathfrak{sl}_2)=\langle L,W,W^4,W^5\rangle .
\]
\item If $k=-17/16$ or $k=-107/64$, then
\[
  N^k(\mathfrak{sl}_2)=\langle L,W,Z^4,Z^5\rangle .
\]
\item If $k=-2$, then
\[
  N^{-2}(\mathfrak{sl}_2)=\langle S,W^3,W^4,W^5\rangle .
\]
The generators have degrees $2,3,4,5$ in the standard grading defined
above.
\end{enumerate}
\end{theorem}

\begin{corollary}\label{cor:weak-generation-LW}
Assume that $k\notin\{0,-2,-3/2,-4/3\}$.
Then $N^k(\mathfrak{sl}_2)$ is generated by the weight three field
$W=W^3$.
If, in addition, $k\neq2$, there is a surjective homomorphism
\[
 \mathcal W^I_R(c,\lambda)\longrightarrow N^k(\mathfrak{sl}_2),
 \qquad L\longmapsto L,\quad W_3\longmapsto W,
\]
where $c=2(k-1)/(k+2)$ and
$\lambda=(k+1)/((k-2)(3k+4))$,
and $\mathcal W^I_R(c,\lambda)$ is the specialization of the universal
two-parameter vertex algebra $\mathcal W(c,\lambda)$ to the
parafermion curve of \cite[Theorem~7.1]{LinshawWinfty}.
\end{corollary}

\begin{proof}
Let $\mathcal U$ be the vertex subalgebra generated by $W$.  Then we have $L,Z^4=W_{(1)}W,Z^5=W_{(1)}Z^4\in\mathcal U$.
The formulas of Section~\ref{sec:vir-heis-realization}, together with
Theorem~\ref{thm:strong-generation}, show that all strong generators of
$N^k(\mathfrak{sl}_2)$ belong to $\mathcal U$ for every $k$ in the stated
range.  Hence $N^k(\mathfrak{sl}_2)=\mathcal U$.

For $k\neq2$, the parameters $c$ and $\lambda$ lie on the parafermion
curve of \cite[Theorem~7.1]{LinshawWinfty}, and the required
localization is defined.  The homomorphism
$\mathcal W^I_R(c,\lambda)\to N^k(\mathfrak{sl}_2)$ is surjective by
\cite[Remark~5.1]{LinshawWinfty}.
\end{proof}

\begin{lemma}\label{lem:zhu-commutative-uniform}
For every $k\neq0$, Zhu's algebra $A(N^k(\mathfrak{sl}_2))$ is
commutative. For $k=-2$, the algebra $A(N^{-2}(\mathfrak{sl}_2))$ is
defined with respect to the standard nonnegative grading above. Consequently,
$A(N_k(\mathfrak{sl}_2))$ is commutative.
\end{lemma}

\begin{proof}
If $k\notin\{0,-2,-3/2,-4/3\}$,
Corollary~\ref{cor:weak-generation-LW} shows that $N^k(\mathfrak{sl}_2)$ is
generated by $W$ and it is a quotient of  $\mathcal W^I_R(c,\lambda)$.   Hence $A(N^k(\mathfrak{sl}_2))$ is generated by
$[W]$ and is commutative by  \cite[Theorem 5.4]{LinshawWinfty}.

At $k=-3/2$, the OPE formulas specialize to
$W^5=\frac{11}{63}W_{(1)}W^4$.
Thus $A(N^{-3/2}(\mathfrak{sl}_2))$ is generated by
$[L],[W],[W^4]$. The class $[L]$ is central. The OPE calculations in
\cite{DLY} are identities in the formal
parameter $k$, and the coefficients used below are regular at
$k=-\frac32,-\frac43$.  The calculation preceding
\cite[Lemma~2.6]{DLY} gives
$[W]*[W^4]-[W^4]*[W]=0$.
Therefore $A(N^{-3/2}(\mathfrak{sl}_2))$ is commutative.  Similarly, at
$k=-4/3$ one has $W^4=\frac{13}{48}W_{(1)}W$,
so $A(N^{-4/3}(\mathfrak{sl}_2))$ is generated by
$[L],[W],[W^5]$, and
\cite[Lemma~2.6]{DLY} gives $[W]*[W^5]-[W^5]*[W]=0$.

Finally, at $k=-2$ the formal OPE identities specialize to
$W^4=-\frac{5}{24}W_{(1)}W$ and
$W^5=-\frac{7}{240}W_{(1)}W^4$.
By Theorem~\ref{thm:strong-generation}, $N^{-2}(\mathfrak{sl}_2)$ is therefore
generated by $S$ and $W$. Since $S$ is central,
$A(N^{-2}(\mathfrak{sl}_2))$ is generated by the commuting elements
$[S]$ and $[W]$. Since $A(N_k(\mathfrak{sl}_2))$ is a quotient of
$A(N^k(\mathfrak{sl}_2))$, it is also commutative.
\end{proof}

\subsection{Proof of Theorem~\ref{thm:strong-generation}}

Consider the vertex  subalgebra
\[
V^k(\mathfrak{sl}_2)(0)
=\{v\in V^k(\mathfrak{sl}_2)\mid h(0)v=0\}.
\]
Since  $k\neq0$ we have
\begin{equation}\label{eq:factor}
V^k(\mathfrak{sl}_2)(0)\cong
\mathcal H^k\otimes N^k(\mathfrak{sl}_2),
\end{equation}
it  is therefore enough to determine a set of generators for
$V^k(\mathfrak{sl}_2)(0)$.

For $k\neq0,-2$, put
\[
B_k=
\begin{cases}
\langle h,L,W,Z^4,Z^5\rangle, & k=-17/16\text{ or }k=-107/64,\\
\langle h,L,W,W^4,W^5\rangle, & \text{otherwise},
\end{cases}
\]
as a subalgebra of $V^k(\mathfrak{sl}_2)(0)$.  For $k=-2$, put
$B_{-2}=\langle h,S,W,W^4,W^5\rangle
\subset V^{-2}(\mathfrak{sl}_2)(0)$.
We shall show
\begin{equation}\label{eq:B-equals-zerocharge}
B_k=V^k(\mathfrak{sl}_2)(0)\qquad(k\neq0).
\end{equation}
Theorem~\ref{thm:strong-generation} then follows from \eqref{eq:factor}.

Let $v_n=f(-n)e(-1)\one$ for $n\geq1$.
The following result is proved in \cite[Theorem~2.1]{DLWY}.

\begin{lemma}\label{lem:pbw-reduction}
If $h$ and all vectors $v_n$, $n\geq1$, lie in a vertex subalgebra
$B\subset V^k(\mathfrak{sl}_2)(0)$, then $B=V^k(\mathfrak{sl}_2)(0)$.
\end{lemma}

Thus it remains to prove that all $v_n$ belong to $B_k$.

We first consider the noncritical levels. For $k\neq0,-2$, substitution
of the formulas for the chosen generators gives
\begin{equation}\label{eq:initial-noncrit}
v_1,v_2,v_3,v_4\in B_k.
\end{equation}
The next lemma follows directly from the proof of \cite[Theorem~2.1]{DLWY}.

\begin{lemma}\label{lem:main-induction}
Let $k\neq0$ and $n\geq3$. Assume that
$v_1,\ldots,v_{n-1}\in B_k$. Then
\begin{equation}
 (v_2)_{(1)}v_{n-1}
 \equiv (n+1)\bigl(n+(n-1)k\bigr)v_n\pmod{B_k}.  \label{eq:C2}
\end{equation}
Consequently, if $n+(n-1)k\neq0$, then $v_n\in B_k$.
\end{lemma}

\begin{proof}
The calculation in the proof of \cite[Theorem~2.1]{DLWY} does not use
the assumption that $k$ is a positive integer and gives the stated
congruence also for $k=-2$. Since $v_2,v_{n-1}\in B_k$, the conclusion
follows.
\end{proof}

The coefficient \eqref{eq:C2} vanishes only when
$k=-n/(n-1)$.
The cases $n=3,4$ are covered by the generators $W^4,W^5$.  For $n\geq5$
and $k=-n/(n-1)$, we need to  use the following alternative approach:

\begin{lemma}\label{lem:repair}
Let $n\geq 5$ and $k\neq 0,-2$. Assume that
$v_1,\ldots,v_{n-1}\in B_k$. Then
\[
 (v_3)_{(1)}v_{n-2}\equiv C_n^{(3)}(k)v_n\pmod{B_k},
 \qquad
 C_n^{(3)}(k)
 =\frac{(n^2-1)\bigl(2n+3k(n-2)\bigr)}{6}.
\]
In particular,
\[
 C_n^{(3)}\left(-\frac{n}{n-1}\right)
 =\frac{n(n+1)(4-n)}{6}\neq 0.
\]
\end{lemma}

\begin{proof}
Set $x_2=f(-(n-1))e(-2)\mathbf1$ and
$x_3=f(-(n-2))e(-3)\mathbf1$. Then we have
\[
 (v_3)_{(1)}v_{n-2}
 \equiv A_n(k)v_n-2x_2+(2+3k)x_3\pmod{B_k},
\]
where
\[
 A_n(k)=2\binom{n-1}{3}
 +\binom{n-1}{2}\bigl(2+(n-2)k\bigr).
\]

Since $v_{n-1},v_{n-2}\in B_k$ and $B_k$ is invariant under the
translation operator $D$, we have
$Dv_{n-1}=(n-1)v_n+x_2\in B_k$. Hence
$x_2\equiv-(n-1)v_n\pmod{B_k}$. Similarly,
\[
 Dv_{n-2}=(n-2)v_{n-1}+f(-(n-2))e(-2)\mathbf1
\]
implies that $f(-(n-2))e(-2)\mathbf1\in B_k$. Applying $D$ to this
vector gives $(n-2)x_2+2x_3\in B_k$, and therefore
$x_3\equiv\binom{n-1}{2}v_n\pmod{B_k}$.

Substitution gives the required congruence, since
\[
 A_n(k)+2(n-1)+(2+3k)\binom{n-1}{2}
 =\frac{(n^2-1)\bigl(2n+3k(n-2)\bigr)}{6}.
\]
The last assertion follows by setting $k=-n/(n-1)$.
\end{proof}

We complete the proof of the theorem.

\begin{proof}[Proof of Theorem~\ref{thm:strong-generation}]
Assume first that $k\neq0,-2$. By \eqref{eq:initial-noncrit},
$v_1,v_2,v_3,v_4\in B_k$. Let $n\geq5$ and assume that
$v_1,\ldots,v_{n-1}\in B_k$. If $n+(n-1)k\neq0$, then
Lemma~\ref{lem:main-induction} gives $v_n\in B_k$. Otherwise,
$k=-n/(n-1)$, and Lemma~\ref{lem:repair} gives the same conclusion,
since
\[
 C_n^{(3)}\left(-\frac{n}{n-1}\right)
 =\frac{n(n+1)(4-n)}6\neq0.
\]
Thus $v_n\in B_k$ for every $n\geq1$. It follows from
Lemma~\ref{lem:pbw-reduction} that
$B_k=V^k(\mathfrak{sl}_2)(0)$. Together with \eqref{eq:factor}, this
proves \textup{(i)} and \textup{(ii)}.

Let now $k=-2$. Applying the formulas for
$S,W,W^4,W^5$ gives $v_1,v_2,v_3,v_4\in B_{-2}$. For $n\geq5$, the
coefficient in Lemma~\ref{lem:main-induction} is
$(n+1)(2-n)\neq0$. Hence that lemma implies by induction that
$v_n\in B_{-2}$ for every $n\geq1$. Lemma~\ref{lem:pbw-reduction}
therefore gives $B_{-2}=V^{-2}(\mathfrak{sl}_2)(0)$, and
\eqref{eq:factor} proves \textup{(iii)}.
\end{proof}

\section{Embedding of $N^k(\mathfrak{sl}_2)$ into
$V^{\Vir}(c_k,0) \otimes \mathcal{H}$}
\label{sec:vir-heis-realization}
We recall the inverse quantum Hamiltonian reduction realization from
\cite{A2019} and derive the induced embedding of
$N^k(\mathfrak{sl}_2)$ into $V^{\Vir}(c_k,0)\otimes\mathcal H$.

Assume that $k\neq0,-2$.  Let
\[
\langle \mu,\mu\rangle=\frac{k}{2},\qquad
\langle \nu,\nu\rangle=-\frac{k}{2},\qquad
\langle \mu,\nu\rangle=0,
\qquad
\mathbf c=\frac{2}{k}(\mu-\nu),
\]
and set
$\Pi(0)=M_{\mu,\nu}(1)\otimes\mathbb C[\mathbb Z\mathbf c]$.
The homomorphism constructed in \cite[Proposition~3.1]{A2019}
\begin{equation}
\Phi:V^k(\mathfrak{sl}_2)\longrightarrow
V^{\Vir}(c_k,0)\otimes \Pi(0),
\qquad
c_k=1-\frac{6(k+1)^2}{k+2}, \label{inverse-embed}
\end{equation}
is determined by
\[
\Phi(e)=e^{\mathbf c},\qquad
\Phi(h)=2\mu(-1),
\]
\[
\Phi(f)=
\Bigl((k+2)T-\nu(-1)^2-(k+1)\nu(-2)\Bigr)e^{-\mathbf c},
\]
where $T$ is a Virasoro vector in  $V^{\Vir}(c_k,0)$.
Writing $b=\nu$, one has
\[
\Phi\bigl(N^k(\mathfrak{sl}_2)\bigr)
\subset V^{\Vir}(c_k,0)\otimes\mathcal H,
\]
where $\mathcal H$ is the Heisenberg vertex algebra generated by the
field $b(z)$. Note that
$b(z)b(w)\sim-\frac{k}{2}(z-w)^{-2}$.

Since $\Phi$ is injective, we identify the fields of
$V^k(\mathfrak{sl}_2)$ and $N^k(\mathfrak{sl}_2)$ with their images under
$\Phi$.  In particular,
\[
L=T-\frac1k:bb:-\partial b
\]
and
\[
\begin{aligned}
W={}&4(4+3k):bbb:
+6k(4+3k):(\partial b)b:
+k^2(4+3k)\partial^2b \\
&-12k(k+2):Tb:
-3k^2(k+2)\partial T.
\end{aligned}
\]
We set $Z^4=W_{(1)}W$ and $Z^5=W_{(1)}Z^4$.

For $k\neq-\frac32$, the primary generator $W^4$ is given by
\[
\begin{aligned}
W^4={}&\frac{17+16k}{36k(3+2k)}\,Z^4
-\frac{8(k-2)k^2(k+2)^2(4+3k)}{3+2k}:LL: \\
&+\frac{9(k-2)k^2(k+2)(4+3k)}{4(3+2k)}\,\partial^2L.
\end{aligned}
\]
For $k\neq-\frac43,-\frac{17}{16}$, one has
\[
\begin{aligned}
W^5={}&-\frac{107+64k}{12k(4+3k)(17+16k)}\,
  W_{(1)}W^4 \\
&+\frac{104(k-3)k(k+2)(1+2k)(3+2k)}
  {(4+3k)(17+16k)}:LW: \\
&-\frac{2(k-3)k(1+2k)(3+2k)(7+2k)}
  {(4+3k)(17+16k)}\,\partial^2W.
\end{aligned}
\]
At $k=-\frac{3}{2}$, the field $W^4$ is
\[
\begin{aligned}
W^4={}&21:bbbb:
-63:(\partial b)bb:
+21:(\partial^2b)b:
+\frac{63}{4}:(\partial b)(\partial b):
-\frac{21}{8}\partial^3b \\
&+\frac{63}{2}:(\partial b)T:
-63:bbT:
+\frac{63}{2}:b\partial T:
-\frac{63}{4}:TT:.
\end{aligned}
\]
At $k=-\frac{4}{3}$, one has $W^4=\frac{13}{48}Z^4$ and
\[
\begin{aligned}
W^5=-\frac{320}{2187}\Big(&
1053:b^5:
-108:Tb^3:
+2196:T^2b:
-2196:Tb\,\partial b:
+516:T\partial^2b: \\
&-3510:b^3\partial b:
+1755:b(\partial b)^2:
-732:T\partial T:
+1206:b^2\partial T:
-42:(\partial T)(\partial b): \\
&+1170:b^2\partial^2b:
-390:(\partial b)(\partial^2b):
-972:b\partial^2T:
-195:b\partial^3b: \\
&+72\partial^3T
+13\partial^4b
\Big).
\end{aligned}
\]
At $k=-\frac{17}{16}$ and $k=-\frac{107}{64}$, we use the generators
$L,W,Z^4,Z^5$.

\section{Highest weight $N^k(\mathfrak{sl}_2)$--modules from
$V^{\Vir}(c_k,0)\otimes\mathcal H$--modules}
\label{sec:evaluation}
In this section we compute the highest weight values of the strong
generators on the two-parameter family of modules used below.

For $(x,y)\in\mathbb C^2$, let $L^{\Vir}(c_k,y)$ be the irreducible
highest weight module for the Virasoro algebra of central charge $c_k$, with highest
weight vector $v^1_y$ satisfying $T(0)v^1_y=yv^1_y$ and
$T(n)v^1_y=0$ for $n>0$. Let $\mathcal F_x$ be the Heisenberg Fock
module generated by $v^2_x$, where $b(0)v^2_x=xv^2_x$ and
$b(n)v^2_x=0$ for $n>0$. Set
$v_{x,y}=v^1_y\otimes v^2_x\in
L^{\Vir}(c_k,y)\otimes\mathcal F_x$.
Then
\[
  L(0)v_{x,y}=f_2(k,x,y)v_{x,y},\qquad
  W(0)v_{x,y}=f_3(k,x,y)v_{x,y},
\]
and, for the primary generators,
\[
  W^4(0)v_{x,y}=f_4(k,x,y)v_{x,y},\qquad
  W^5(0)v_{x,y}=f_5(k,x,y)v_{x,y}.
\]
The first two polynomials are $f_2(k,x,y)=x-\frac{x^2}{k}+y$ and
\[
f_3(k,x,y)=2(k-2x)\bigl(-4x^2+3k^2(x+y)+k(4x-3x^2+6y)\bigr).
\]
The remaining zero-mode eigenvalues are
\begin{align*}
f_{4}(k,x,y)
&=
-2\Bigl(
48x^4
+4kx^2(1-24x+34x^2-24y) \\
&\quad
+6k^5(x^2+(-3+y)y+x(-3+2y)) \\
&\quad
+k^4(-12x^3+x^2(101-12y)+y(-55+19y)+3x(-13+34y)) \\
&\quad
+k^2(-272x^3+83x^4+x^2(71-204y)-4y(1+5y)+x(-4+96y)) \\
&\quad
+k^3(-166x^3+6x^4+x^2(175-102y)+4(-10+y)y+x(-23+204y))
\Bigr),
\end{align*}
\begin{align*}
f_{5}(k,x,y)
&=
2(k-2x)\Bigl(
192x^4
+8kx^2(5-48x+68x^2-60y) \\
&\quad
+10k^5(5x^2+y(-7+5y)+x(-7+10y)) \\
&\quad
+k^2(-1088x^3+345x^4+x^2(353-860y)-20y(3+7y)+40x(-1+12y)) \\
&\quad
+k^4(-100x^3-5x^2(-83+20y)+5y(-53+33y)+x(-197+510y)) \\
&\quad
+k^3(-690x^3+50x^4+x^2(741-510y)+20y(-14+3y)+x(-161+860y))
\Bigr).
\end{align*}

At the levels $k=-\frac{17}{16}$ and $k=-\frac{107}{64}$, the
$A(N^k(\mathfrak{sl}_2))$-character is determined by its values on the
generators $L,W,Z^4,Z^5$. Put
$g_2=f_2$ and $g_3=f_3$.
The zero-mode eigenvalues of $Z^4$ and $Z^5$ on $v_{x,y}$ are
\begin{align}
 g_4(k,x,y)=18k\Bigl(&
 -64x^4-32kx^2(-1-4x+3x^2-4y)+15k^5(x+y) \notag\\
&-4k^2\bigl(-48x^3+9x^4+x^2(14-36y)+4y(2+y)+8x(1+4y)\bigr)\notag\\
&+k^4\bigl(-51x^2-4(-10+y)y-8x(-4+5y)\bigr) \notag\\
&+4k^3\bigl(18x^3+y-4y^2+2x^2(-16+5y)-2x(1+18y)\bigr)
\Bigr),
\label{eq:g4-main}
\end{align}
\begin{align}
 g_5(k,x,y)=216k^2(k-2x)\Bigl(&
 -256x^4-64kx^2(-9-8x+9x^2-10y)+105k^6(x+y)\notag\\
&+k^5\bigl(-213x^2+2(181-14y)y-4x(-71+34y)\bigr)\notag\\
&-16k^2\bigl(-72x^3+27x^4+6y(9+2y)+4x(9+10y)-x^2(37+72y)\bigr)\notag\\
&-4k^3\bigl(-216x^3+27x^4+2y(129+38y)+4x(53+72y)\notag\\
&\qquad\qquad -2x^2(-57+86y)\bigr)\notag\\
&+4k^4\bigl(54x^3+y-40y^2+x^2(-179+34y)-2x(15+86y)\bigr)
\Bigr).
\label{eq:g5-main}
\end{align}

\begin{proposition}\label{prop:existence-Lxy}
Assume that $k\neq0,-2$. For each $(x,y)\in\mathbb C^2$, there exists an
irreducible highest weight $N^k(\mathfrak{sl}_2)$--module $L_k[x,y]$
generated by a highest weight vector $w_{x,y}$. If
$k\neq-17/16,-107/64$, the character $\chi^k_{x,y}$ is determined by its values on
$[L],[W],[W^4],[W^5]$:
\[
\begin{aligned}
L(0)w_{x,y}&=f_2(k,x,y)w_{x,y},
& W(0)w_{x,y}&=f_3(k,x,y)w_{x,y},\\
W^4(0)w_{x,y}&=f_4(k,x,y)w_{x,y},
& W^5(0)w_{x,y}&=f_5(k,x,y)w_{x,y}.
\end{aligned}
\]
At either exceptional level, the character $\chi^k_{x,y}$ is determined
by its values on $[L],[W],[Z^4],[Z^5]$:
\[
\begin{aligned}
L(0)w_{x,y}&=g_2(k,x,y)w_{x,y},
& W(0)w_{x,y}&=g_3(k,x,y)w_{x,y},\\
Z^4(0)w_{x,y}&=g_4(k,x,y)w_{x,y},
& Z^5(0)w_{x,y}&=g_5(k,x,y)w_{x,y}.
\end{aligned}
\]
\end{proposition}
\begin{proof}
Let $V^k[x,y]=N^k(\mathfrak{sl}_2)v_{x,y}
\subset L^{\Vir}(c_k,y)\otimes\mathcal F_x$.
The preceding zero-mode formulas show that $V^k[x,y]$ is a highest weight
$N^k(\mathfrak{sl}_2)$--module with $A(N^k(\mathfrak{sl}_2))$-character
$\chi^k_{x,y}$.  Its irreducible quotient is denoted by $L_k[x,y]$.
\end{proof}

\begin{remark}
The character values were obtained by direct calculation in Mathematica.
\end{remark}

\section{Relations and presentation of Zhu's algebra
$A(N^k(\mathfrak{sl}_2))$}
\label{sec:zhu-polynomial-system}
\label{sec:construction-relations}
We define the ideals $I_k$ and, at the two exceptional levels,
$I_k^z$, which occur in the presentation of Zhu's algebra. We also
introduce auxiliary ideals $K_k$ and $K_k^z$ obtained from null fields.
In Theorem~\ref{thm:zhu-presentation}, we prove that the presentation
ideal and the corresponding auxiliary ideal both coincide with the
kernel of the natural homomorphism onto Zhu's algebra.
Throughout this section, $k\neq0,-2$.
Put $S=\mathbb C[w_2,w_3]$ and $B=\mathbb C[x,y]$.

Let $V=N^k(\mathfrak{sl}_2)$ and write $[a]=a+O(V)$ for the class of
$a\in V$ in $A(V)$.  By Lemma~\ref{lem:zhu-commutative-uniform} and
Theorem~\ref{thm:strong-generation}, the map
\[
\varphi_k:\mathbb C[w_2,w_3,w_4,w_5]\longrightarrow A(V),
\qquad
w_2\mapsto[L],\quad w_3\mapsto[W],\quad
w_4\mapsto[W^4],\quad w_5\mapsto[W^5],
\]
is surjective whenever the primary generators $W^4,W^5$ are used.
If $k\neq-17/16,-107/64$, define
\[
\begin{aligned}
\psi_k:\mathbb C[w_2,w_3,w_4,w_5]&\longrightarrow B,\\
\psi_k(w_2)&=f_2(k,x,y),\qquad
\psi_k(w_3)=f_3(k,x,y),\\
\psi_k(w_4)&=f_4(k,x,y),\qquad
\psi_k(w_5)=f_5(k,x,y).
\end{aligned}
\]
Proposition~\ref{prop:existence-Lxy} implies that
\[
  \ker\varphi_k\subset\ker\psi_k.
\]

Let $\mathbf v^0,\mathbf v^1,\mathbf v^2$ be the null fields of weights
$8,9,10$ from \cite[Section~2 and Appendix~C]{DLY}.  The calculation in
\cite[Section~2]{DLY} gives
$\varphi_k(R_0)=[\mathbf v^0]$ and
$\varphi_k(R_1)=[\mathbf v^1]$ in $A(V)$, where $R_0=Q_0$ and
$R_1=Q_1$ in the notation of
\cite{DLY}.
The calculation in the proof of \cite[Lemma~2.6]{DLY} also gives a
polynomial $T_2\in\mathbb C[w_2,w_3,w_4,w_5]$ such that
$\varphi_k(T_2)=[\mathbf v^2]$ in $A(V)$ and
\[
  T_2=-w_5^2+U_2,
  \qquad
  U_2\in
  \mathbb C[w_2,w_3]+\mathbb C[w_2,w_3]w_4+
  \mathbb C[w_2,w_3]w_5+\mathbb C[w_2,w_3]w_4^2.
\]
Since $\mathbf v^0,\mathbf v^1,\mathbf v^2$ are null fields,
$K_k=(R_0,R_1,T_2)\subset\ker\varphi_k$.
The identity $T_2=-w_5^2+U_2$, together with the leading
terms of $R_0$ and $R_1$, implies that the quotient by $K_k$ is generated over
$\mathbb C[w_2,w_3]$ by $1,w_4,w_5$.
The polynomial $T_2$ is auxiliary; the final presentation uses the
polynomial $R_2$ defined below.

Let $R_2$ be the polynomial given in Appendix~\ref{app-a}, and put
$I_k=(R_0,R_1,R_2)$.
Direct substitution gives
\[
R_i\bigl(k,f_2(k,x,y),f_3(k,x,y),
f_4(k,x,y),f_5(k,x,y)\bigr)=0,
\qquad 0\leq i\leq2.
\]
Thus $I_k$ is contained in the kernel of the polynomial evaluation map
defined above.

We use the normalizations
$\operatorname{coeff}_{w_4^2}R_0=2(64k+107)$,
$\operatorname{coeff}_{w_4w_5}R_1=4$, and
$\operatorname{coeff}_{w_5^2}R_2=-1$.

At the two levels $k=-17/16,-107/64$,
the formulas involving the primary generators $W^4,W^5$ are singular, and
we use instead the generators $Z^4,Z^5$ of weights $4,5$.

By Lemma~\ref{lem:zhu-commutative-uniform} and
Theorem~\ref{thm:strong-generation}, $A(N^k(\mathfrak{sl}_2))$ is
commutative and is generated by
$[L]$, $[W]$, $[Z^4]$, $[Z^5]$.
Let $w_2,w_3,z_4,z_5$ be independent variables.  Then there is a
surjective homomorphism
\[
\varphi_k^z:\mathbb C[w_2,w_3,z_4,z_5]\longrightarrow A(V),
\qquad
w_2\mapsto [L],\quad w_3\mapsto [W],\quad
z_4\mapsto [Z^4],\quad z_5\mapsto [Z^5].
\]
Thus $z_4$ and $z_5$ are variables in the polynomial ring, whereas
$[Z^4]$ and $[Z^5]$ are their images under $\varphi_k^z$.
Define
\[
\begin{aligned}
\psi_k^z:\mathbb C[w_2,w_3,z_4,z_5]&\longrightarrow B,\\
\psi_k^z(w_2)&=g_2(k,x,y),\qquad
\psi_k^z(w_3)=g_3(k,x,y),\\
\psi_k^z(z_4)&=g_4(k,x,y),\qquad
\psi_k^z(z_5)=g_5(k,x,y). &
\end{aligned}
\]
Proposition~\ref{prop:existence-Lxy} implies that
$
  \ker\varphi_k^z\subset\ker\psi_k^z
$.

\begin{lemma}\label{lem:exceptional-null-fields}
Let $k=-17/16$ or $k=-107/64$.  There are nonzero null fields
$P_8^z,P_9^z,P_{10}^z$ of weights $8,9,10$, respectively, written in the
generators $L,W,Z^4,Z^5$.  They may be normalized so that the coefficients of
$:Z^4Z^4:$, $:Z^4Z^5:$ and $:Z^5Z^5:$ are, respectively,
$2(3k+4)$, $1$ and $-1$,
and the coefficient of $:WZ^5:$ in $P_8^z$ is $-3(2k+3)$.
\end{lemma}

\begin{proof}
The calculation in \cite[Section~2 and Appendix~C]{DLY} gives null
fields of conformal weights $8,9,10$ in $N^k(\mathfrak{sl}_2)$, and the
same calculation applies to the exceptional levels.  At these levels
$N^k(\mathfrak{sl}_2)$ is strongly generated by $L,W,Z^4,Z^5$.

Whenever the denominators in the following formulas are nonzero, the
two systems of generators are related by
\[
\begin{aligned}
W^4={}&\frac{17+16k}{36k(3+2k)}\,Z^4
-\frac{8(k-2)k^2(k+2)^2(4+3k)}{3+2k}:LL: \\
&+\frac{9(k-2)k^2(k+2)(4+3k)}{4(3+2k)}\,\partial^2L
\end{aligned}
\]
and
\[
\begin{aligned}
W^5={}&-\frac{107+64k}{12k(4+3k)(17+16k)}\,W_{(1)}W^4 \\
&+\frac{104(k-3)k(k+2)(1+2k)(3+2k)}
  {(4+3k)(17+16k)}:LW: \\
&-\frac{2(k-3)k(1+2k)(3+2k)(7+2k)}
  {(4+3k)(17+16k)}\,\partial^2W .
\end{aligned}
\]
Using these formulas, together with
$Z^4=W_{(1)}W$ and $Z^5=W_{(1)}Z^4$,
we rewrite the null fields from \cite{DLY} in terms of
$L,W,Z^4,Z^5$.  The substitution is made before specializing $k$.
After clearing the denominators and removing common scalar factors, the
resulting expressions are regular at
$k=-17/16,-107/64$.
Their specializations give the null fields
$P_8^z,P_9^z,P_{10}^z$.

We normalize these fields so that the coefficients of
$:Z^4Z^4:$, $:Z^4Z^5:$ and $:Z^5Z^5:$ are, respectively,
$
2(3k+4), 1, -1$,
and the coefficient of $:WZ^5:$ in $P_8^z$ is $-3(2k+3)$.
These coefficients are nonzero at the two exceptional levels, so the
specialized null fields are nonzero.
\end{proof}

\begin{remark}\label{rem:P10-supplement}
An explicit formula for the weight-ten null field $P_{10}^z$ in the generators
$L,W,Z^4,Z^5$ is available at
\href{https://www.dropbox.com/scl/fi/dm0o130us5l05dl4xj7wb/P10-null-vector.pdf?rlkey=t8ofxxicvsqgkmxftod251nwb&dl=0}{P10-null-vector.pdf}.
In that formula, the coefficient of $:Z^5Z^5:$ is normalized to one.
Multiplication by $-1$ gives the normalization used in
Lemma~\ref{lem:exceptional-null-fields}.
\end{remark}

\medskip
Let $k=-17/16$ or $k=-107/64$.  Choose polynomials
$F_8^z,F_9^z,F_{10}^z\in\mathbb C[w_2,w_3,z_4,z_5]$ such that
\[
  \varphi_k^z(F_8^z)=[P_8^z],\qquad
  \varphi_k^z(F_9^z)=[P_9^z],\qquad
  \varphi_k^z(F_{10}^z)=[P_{10}^z]
\]
in $A(V)$.
The normal-form calculation in the proof of \cite[Lemma~2.6]{DLY},
together with Lemma~\ref{lem:exceptional-null-fields}, gives
\begin{align*}
F_8^z&=2(3k+4)z_4^2+U_8^z,
&U_8^z&\in
\mathbb C[w_2,w_3]+\mathbb C[w_2,w_3]z_4+
\mathbb C[w_2,w_3]z_5,\\
F_9^z&=z_4z_5+U_9^z,
&U_9^z&\in
\mathbb C[w_2,w_3]+\mathbb C[w_2,w_3]z_4+
\mathbb C[w_2,w_3]z_5,\\
F_{10}^z&=-z_5^2+U_{10}^z,
&U_{10}^z&\in
\mathbb C[w_2,w_3]+\mathbb C[w_2,w_3]z_4+
\mathbb C[w_2,w_3]z_5+\mathbb C[w_2,w_3]z_4^2.
\end{align*}
Since $P_8^z,P_9^z,P_{10}^z$ are null fields,
$K_k^z=(F_8^z,F_9^z,F_{10}^z)\subset\ker\varphi_k^z$.

Let $I_k^z=(R_0^z,R_1^z,R_2^z)$,
where the polynomials are given in Appendix~\ref{app-a}. Direct
substitution gives
\[
R_i^z\bigl(
g_2(k,x,y),g_3(k,x,y),g_4(k,x,y),g_5(k,x,y)
\bigr)=0
\qquad (0\leq i\leq2)
\]
for all $(x,y)\in\mathbb C^2$.  Thus $I_k^z$ is contained in the
kernel of the exceptional evaluation map. The triangular forms of
$F_8^z,F_9^z,F_{10}^z$ imply that the quotient by $K_k^z$ is generated
over $\mathbb C[w_2,w_3]$ by $1,z_4,z_5$.

The two systems are related by the triangular change of variables
\begin{align}
w_4={}&
\frac{16k+17}{36k(2k+3)}z_4
-\frac{8k^2(k-2)(k+2)^2(3k+4)}{2k+3}w_2^2 \notag\\
&-\frac{k^2(k-2)(k+2)(3k+4)(32k+37)}{2(2k+3)}w_2,
\label{eq:w4-z4-main}
\end{align}
\begin{align}
 w_5={}&
 -\frac{64k+107}{432k^2(2k+3)(3k+4)}z_5 \notag\\
&+\frac{4k(k+2)(25k^3+16k^2-126k-142)}
{(2k+3)(3k+4)}w_2w_3 \notag\\
&+\frac{3k(560k^4+1467k^3-1724k^2-7532k-5296)}
  {4(2k+3)(3k+4)}w_3 .
\label{eq:w5-z5-main}
\end{align}
Substitution gives
\begin{align}
R_0 &\mapsto
\frac{(16k+17)^2(64k+107)}
  {1296k^2(2k+3)^2(3k+4)}\,R_0^z,\notag\\[2mm]
R_1 &\mapsto
-\frac{(16k+17)(64k+107)}
  {3888k^3(2k+3)^2(3k+4)}\,R_1^z,\notag\\[2mm]
R_2 &\mapsto
\frac{(64k+107)^2}
  {186624k^4(2k+3)^2(3k+4)^2}\,R_2^z
  +A(k,w_2)R_0^z,
\label{eq:comparison-relz-main}
\end{align}
where
\begin{equation}\label{eq:A-correction-main}
\begin{aligned}
A(k,w_2)={}&-
\frac{64k+107}{5184k(2k+3)^3(3k+4)^2}
\Bigl(3(560k^4+1467k^3 \\
&\quad -1724k^2-7532k-5296)
+32(k+2)(25k^3+16k^2 \\
&\quad -126k-142)w_2\Bigr).
\end{aligned}
\end{equation}
Thus the two systems of relations are equivalent whenever the
coefficients of $R_0^z,R_1^z,R_2^z$ in
\eqref{eq:comparison-relz-main} are nonzero. At $k=-17/16$ and
$k=-107/64$, we use the relations associated with the generators
$L,W,Z^4,Z^5$.

We regard $B$ as an $S$-algebra by
$w_2\mapsto f_2(k,x,y)$ and $w_3\mapsto f_3(k,x,y)$.  Put
\begin{equation}\label{eq:cubic-main}
C_{w_2,w_3}(t)=8t^3-12kt^2+
\bigl(4k^2+12k(k+2)w_2\bigr)t
+w_3-6k^2(k+2)w_2.
\end{equation}
The homomorphism
\[
S[t]/(C_{w_2,w_3}(t))\longrightarrow B,\qquad t\longmapsto x,
\]
is an $S$-algebra isomorphism
\begin{equation}\label{eq:cubic-ring-isomorphism}
  S[t]/(C_{w_2,w_3}(t))\cong B.
\end{equation}
Its inverse is determined by
$x\mapsto t$ and $y\mapsto w_2-t+t^2/k$.
In particular, $B$ is a free $S$-module with basis $1,t,t^2$.

\begin{lemma}\label{lem:kernel-criterion}
Let $P=S[u,v]$, and let $\psi:P\to B$ be an $S$-algebra homomorphism.
Assume that $1,\psi(u),\psi(v)$ are linearly independent over $S$.
If an ideal $J\subset\ker\psi$ has the property that $P/J$ is generated
over $S$ by the residue classes of $1,u,v$, then
\[
  J=\ker\psi,
\]
and $P/J$ is a free $S$-module with basis given by these residue
classes.
\end{lemma}

\begin{proof}
There is a surjective $S$-linear map
\[
S^3\longrightarrow P/J,\qquad
(a,b,c)\longmapsto a+bu+cv.
\]
Its composition with the map $P/J\to B$ induced by $\psi$ is injective
by the assumed linear independence.  Hence the first map is an
isomorphism and $P/J\to B$ is injective.  Therefore $J=\ker\psi$.
\end{proof}

\begin{lemma}\label{lem:character-value-independence}
\begin{enumerate}[label=\textup{(\roman*)}]
\item If $k\neq0,-2,-17/16,-107/64$, the elements
$1,f_4(k,x,y),f_5(k,x,y)$
are linearly independent over $S$ in $B$.
\item If $k=-17/16$ or $k=-107/64$, the elements
$1,g_4(k,x,y),g_5(k,x,y)$
are linearly independent over $S$ in $B$.
\end{enumerate}
\end{lemma}

\begin{proof}
Reduce the last two elements in each case in the basis $1,t,t^2$ from
\eqref{eq:cubic-ring-isomorphism}. The determinants of their
coefficients of $t$ and $t^2$ are
\[
2(16k+17)(64k+107)D_k(w_2,w_3)
\]
in the first case and
\[
-31104k^3(2k+3)^2(3k+4)D_k(w_2,w_3)
\]
in the second case, where
\[
D_k(w_2,w_3)=w_3^2-4w_2^2k^4(k+2)^2
+32w_2^3k^3(k+2)^3.
\]
In \textup{(i)}, the scalar factor is nonzero, and
$D_k(w_2,w_3)$ is a nonzero polynomial. Hence the first determinant is
a nonzero element of $S$. At each exceptional level, the scalar factor
in the second determinant is nonzero. Hence the second determinant is
also a nonzero element of $S$. Since $S$ is an integral domain, these
nonzero determinants imply the asserted $S$-linear independence. The
reductions are given in
Appendix~\ref{app-b}.
\end{proof}

\begin{theorem}\label{thm:zhu-presentation}
Let $k\neq0,-2$.
\begin{enumerate}[label=\textup{(\roman*)}]
\item If $k\neq-17/16,-107/64$, then
\[
I_k=\ker\psi_k=\ker\varphi_k
\]
and
\[
A(N^k(\mathfrak{sl}_2))\cong
\frac{\mathbb C[w_2,w_3,w_4,w_5]}{(R_0,R_1,R_2)}.
\]
This algebra is a free $S$-module with basis given by the residue
classes of $1,w_4,w_5$.
\item If $k=-17/16$ or $k=-107/64$, then
\[
I_k^z=\ker\psi_k^z=\ker\varphi_k^z
\]
and
\[
A(N^k(\mathfrak{sl}_2))\cong
\frac{\mathbb C[w_2,w_3,z_4,z_5]}
{(R_0^z,R_1^z,R_2^z)}.
\]
This algebra is a free $S$-module with basis given by the residue
classes of $1,z_4,z_5$.
\end{enumerate}
\end{theorem}

\begin{proof}
In the first case, the relation $R_0$ expresses $w_4^2$ as an
$S$-linear combination of $1,w_4,w_5$.  The relation $R_1$
expresses $w_4w_5$ in the same form, after reducing $w_4^2$.
Finally, $R_2$ expresses $w_5^2$ in this form.  Hence
$\mathbb C[w_2,w_3,w_4,w_5]/I_k$ is generated over $S$ by
$1,w_4,w_5$, by induction on the total degree in
$w_4,w_5$.  Since $I_k\subset\ker\psi_k$,
Lemmas~\ref{lem:kernel-criterion} and
\ref{lem:character-value-independence} give $I_k=\ker\psi_k$.

The polynomials $R_0,R_1,T_2$ express
$w_4^2,w_4w_5,w_5^2$, successively, as $S$-linear combinations of
$1,w_4,w_5$. Hence the quotient by $K_k$ is generated over $S$ by the
residue classes of $1,w_4,w_5$.
Since $K_k\subset\ker\varphi_k\subset\ker\psi_k$,
Lemma~\ref{lem:kernel-criterion} gives
$K_k=\ker\varphi_k=\ker\psi_k=I_k$.

At the two exceptional levels, $2(3k+4)\neq0$.  The polynomials
$R_0^z,R_1^z,R_2^z$ express $z_4^2,z_4z_5,z_5^2$, successively, as
$S$-linear combinations of $1,z_4,z_5$.  Hence the quotient by
$I_k^z$ is generated over $S$ by the residue classes of
$1,z_4,z_5$.  The leading-term formulas for
$F_8^z,F_9^z,F_{10}^z$ in Section~\ref{sec:zhu-polynomial-system}
give the same conclusion for the quotient by $K_k^z$.  Applying
Lemmas~\ref{lem:kernel-criterion} and
\ref{lem:character-value-independence}, first to
$I_k^z\subset\ker\psi_k^z$ and then to
$K_k^z\subset\ker\varphi_k^z\subset\ker\psi_k^z$, gives
$I_k^z=K_k^z=\ker\psi_k^z=\ker\varphi_k^z$.

In part~\textup{(i)}, the equality $\ker\varphi_k=I_k$ induces
\[
A(N^k(\mathfrak{sl}_2))
\cong\mathbb C[w_2,w_3,w_4,w_5]/I_k,
\]
and
Lemma~\ref{lem:kernel-criterion} identifies the residue classes of
$1,w_4,w_5$ as an $S$-basis.  In part~\textup{(ii)}, the equality
$\ker\varphi_k^z=I_k^z$ induces
\[
A(N^k(\mathfrak{sl}_2))
\cong\mathbb C[w_2,w_3,z_4,z_5]/I_k^z,
\]
and the same lemma identifies the residue classes of $1,z_4,z_5$ as
an $S$-basis.
\end{proof}

\section{Irreducible representations of Zhu's algebra
$A(N^k(\mathfrak{sl}_2))$ and irreducible highest weight
$N^k(\mathfrak{sl}_2)$--modules}
\label{sec:zhu-parametrization}
\label{sec:zhu-representations}
In this section, we classify the irreducible representations of Zhu's
algebra $A(N^k(\mathfrak{sl}_2))$ and deduce the classification of
irreducible highest weight $N^k(\mathfrak{sl}_2)$--modules.

\begin{theorem}\label{thm:zhu-parametrization}
Let $k\neq0,-2$.  Every irreducible representation of
$A(N^k(\mathfrak{sl}_2))$ is one-dimensional and is given by the
character $\chi^k_{x,y}$ for some $(x,y)\in\mathbb C^2$.
\end{theorem}

\begin{proof}
By Lemma~\ref{lem:zhu-commutative-uniform} and
Theorem~\ref{thm:strong-generation},
$A(N^k(\mathfrak{sl}_2))$ is a finitely generated
commutative $\mathbb C$-algebra. Hence every irreducible
$A(N^k(\mathfrak{sl}_2))$-module is one-dimensional.

Let $B=\mathbb C[x,y]$ and $S=\mathbb C[w_2,w_3]$. By
Theorem~\ref{thm:zhu-presentation},
$
\ker\varphi_k=\ker\psi_k
$
if $k\neq-17/16,-107/64$, whereas
$
\ker\varphi_k^z=\ker\psi_k^z
$
at the two exceptional levels. Consequently, $\psi_k$,
respectively $\psi_k^z$, induces an injective homomorphism
\[
\iota_k:A(N^k(\mathfrak{sl}_2))\longrightarrow B.
\]
We identify $A(N^k(\mathfrak{sl}_2))$ with its image under
$\iota_k$.

By \eqref{eq:cubic-ring-isomorphism}, $B$ is a free
$S$-module with basis $1,x,x^2$. Since
\[
S\subset A(N^k(\mathfrak{sl}_2))\subset B,
\]
these elements also generate $B$ as an
$A(N^k(\mathfrak{sl}_2))$-module. Thus $B$ is finitely
generated as an $A(N^k(\mathfrak{sl}_2))$-module. It follows
that every element of $B$ satisfies a monic polynomial with
coefficients in $A(N^k(\mathfrak{sl}_2))$. In other words,
$B$ is integral over $A(N^k(\mathfrak{sl}_2))$.

In the present case, this can also be seen directly. The
element $x$ satisfies the monic cubic equation obtained from
\eqref{eq:cubic-main}, and
$
y=w_2-x+\frac{x^2}{k}.
$
Consequently,
$
B=A(N^k(\mathfrak{sl}_2))[x].
$

Let
\[
\chi:A(N^k(\mathfrak{sl}_2))\longrightarrow\mathbb C
\]
be the character of an irreducible representation and put
$\mathfrak m=\ker\chi$. By the lying-over theorem
\cite[Corollary~5.8 and Theorem~5.10]{AtiyahMacdonald},
there is a maximal ideal $\mathfrak n$ of $B$ such that
\[
\mathfrak n\cap A(N^k(\mathfrak{sl}_2))=\mathfrak m.
\]
Every maximal ideal of $B=\mathbb C[x,y]$ has the form
$
\mathfrak n=(x-x_0,y-y_0)
$
for some $(x_0,y_0)\in\mathbb C^2$. For every
$a\in A(N^k(\mathfrak{sl}_2))$, we have
$a-\chi(a)\in\mathfrak m\subset\mathfrak n$. Therefore
\[
\chi(a)=\iota_k(a)(x_0,y_0).
\]
By the definition of $\psi_k$, respectively $\psi_k^z$, the
character on the right-hand side is $\chi^k_{x_0,y_0}$.
This proves the assertion.
\end{proof}

Together with Proposition~\ref{prop:existence-Lxy}, this proves the
classification of irreducible highest weight
$N^k(\mathfrak{sl}_2)$--modules.

\begin{theorem}\label{class-univ}
Let $k\in\mathbb C\setminus\{0,-2\}$. Every irreducible highest weight
$N^k(\mathfrak{sl}_2)$--module is isomorphic to $L_k[x,y]$ for some
$(x,y)\in\mathbb C^2$. If
$k\neq-17/16,-107/64$, the tuple of highest weight eigenvalues
corresponding to $L,W,W^4,W^5$ is
\[
  (f_2(k,x,y),f_3(k,x,y),f_4(k,x,y),f_5(k,x,y)),
\]
whereas, for $k=-17/16$ or $k=-107/64$, the tuple corresponding to
$L,W,Z^4,Z^5$ is
\[
  (g_2(k,x,y),g_3(k,x,y),g_4(k,x,y),g_5(k,x,y)).
\]
\end{theorem}

\section{The critical level $k=-2$}
\label{sec:critical-level}

Let $S(z)=\sum_{n\in\mathbb Z}S(n)z^{-n-2}$
be the critical Segal--Sugawara field. At the critical level,
$N^{-2}(\mathfrak{sl}_2)$ is strongly generated by
$S,W=W^3,W^4,W^5$.
The field $S$ is central.

For $k\neq-2$, set $\Sigma_k=2(k+2)L$. The formula for $L$ gives
\[
\Sigma_k=-h(-2)-\frac1k h(-1)^2+2e(-1)f(-1).
\]
Thus $\Sigma_k$ is regular at $k=-2$, and
\[
\left.\Sigma_k\right|_{k=-2}
=-h(-2)+\frac12h(-1)^2+2e(-1)f(-1)=S.
\]
In the last equality we used
$f(-1)e(-1)=e(-1)f(-1)-h(-2)$.

Let $P_{-2}=\mathbb C[y,w_3,w_4,w_5]$ and
$\mathcal S=\mathbb C[y,w_3]$.
The variables $y,w_3,w_4,w_5$ correspond, respectively, to
$[S],[W],[W^4],[W^5]$. By
Lemma~\ref{lem:zhu-commutative-uniform} and
Theorem~\ref{thm:strong-generation}, there is a surjective homomorphism
\[
\varphi_{-2}:P_{-2}\longrightarrow
A(N^{-2}(\mathfrak{sl}_2)),
\quad
y\mapsto[S],\quad w_3\mapsto[W],\quad
w_4\mapsto[W^4],\quad w_5\mapsto[W^5].
\]

In the polynomials of Section~\ref{sec:zhu-polynomial-system}, replace
$w_2$ by $\sigma/(2(k+2))$, where $\sigma$ represents
$[\Sigma_k]$, and then set $\sigma=y$ and $k=-2$. This gives the
polynomials in the following lemma.

\begin{lemma}\label{lem:critical-universal-relations}
Define the following elements of $P_{-2}$:
\begin{align*}
\mathcal R_0={}&
14w_4^2+308y^2w_4+168yw_4
-75w_3w_5+2325yw_3^2+900w_3^2 \\
&\qquad
-26656y^4-35952y^3-12096y^2,
\end{align*}
\begin{align*}
\mathcal R_1={}&
4w_4w_5-628yw_3w_4-216w_3w_4
+224y^2w_5+144yw_5+315w_3^3 \\
&\qquad
+5152y^3w_3+5616y^2w_3+2304yw_3,
\end{align*}
and
\begin{align*}
\mathcal R_2={}&
5264yw_4^2+1512w_4^2-2205w_3^2w_4
-166432y^3w_4-185808y^2w_4-52416yw_4 \\
&\qquad
-150w_5^2+1800w_3w_5
+219120y^2w_3^2+108720yw_3^2 \\
&\qquad
-426496y^5-26880y^4+741888y^3+387072y^2.
\end{align*}
Then there exists a polynomial
\[
\widetilde{\mathcal R}_2
=-w_5^2+\widetilde U_2,
\qquad
\widetilde U_2\in
\mathcal S+\mathcal S w_4+\mathcal S w_5+\mathcal S w_4^2,
\]
such that
\[
J_{-2}:=
(\mathcal R_0,\mathcal R_1,\widetilde{\mathcal R}_2)
\subset\ker\varphi_{-2}.
\]
\end{lemma}

\begin{proof}
Take the images in Zhu's algebra of the null fields of conformal
weights $8,9,10$ from \cite[Section~2 and Appendix~C]{DLY}. In the
corresponding polynomial identities, replace $[L]$ by
$[\Sigma_k]/(2(k+2))$ and divide by the common powers of $k+2$.
The coefficients of the resulting polynomials are regular at $k=-2$.
The specializations of weights $8$ and $9$ are
$\mathcal R_0$ and $\mathcal R_1$. The polynomial of weight $10$ can
be normalized so that the coefficient of $w_5^2$ is $-1$. Its
remaining terms belong to
$\mathcal S+\mathcal S w_4+\mathcal S w_5+\mathcal S w_4^2$.
This gives $\widetilde{\mathcal R}_2$ and proves the asserted
inclusion.
\end{proof}

At the critical level, the realization from \cite{A2019} gives an
embedding $V^{-2}(\mathfrak{sl}_2)\longrightarrow
\mathcal Z\otimes\Pi(0)$, whose restriction gives an injective vertex
algebra homomorphism
\begin{equation}
N^{-2}(\mathfrak{sl}_2)\longrightarrow \mathcal Z\otimes\mathcal H, \label{critical-2}
\end{equation}
where $\mathcal Z$ is the commutative vertex algebra generated by $S$
and $\mathcal H$ is a rank-one Heisenberg vertex algebra. For
$(x,y)\in\mathbb C^2$, let $\mathbb C_y$ be the one-dimensional
$\mathcal Z$--module on which $S(0)=y\,\operatorname{Id}$ and
$S(n)=0$ for $n\neq0$,
and let $\mathcal F_x$ be the Heisenberg Fock module of highest weight
$x$.
Put $u=x+1$.
On the highest weight vector of $\mathbb C_y\otimes\mathcal F_x$, the
zero modes of $W,W^4,W^5$ act by
\begin{align}
w_3(x,y)&=-8u^3+(12y+8)u,\label{eq:critical-w3}\\
w_4(x,y)&=
2\bigl(17y^2+60yu^2+12y-60u^4+60u^2\bigr),
\label{eq:critical-w4}\\
w_5(x,y)&=
4u\bigl(135y^2-20yu^2+140y-84u^4+60u^2+24\bigr).
\label{eq:critical-w5}
\end{align}
Let $L_{-2}[x,y]$ denote the irreducible quotient of the cyclic
$N^{-2}(\mathfrak{sl}_2)$--module generated by this vector.

\begin{theorem}\label{thm:critical-zhu-presentation}
Let $I_{-2}=(\mathcal R_0,\mathcal R_1,\mathcal R_2)\subset P_{-2}$.
Then
\[
I_{-2}=J_{-2}=\ker\psi_{-2}=\ker\varphi_{-2},
\]
where $\psi_{-2}:P_{-2}\to\mathbb C[u,y]$ is defined by
\begin{align*}
\psi_{-2}(y)&=y,\\
\psi_{-2}(w_3)&=-8u^3+(12y+8)u,\\
\psi_{-2}(w_4)&=
2(17y^2+60yu^2+12y-60u^4+60u^2),\\
\psi_{-2}(w_5)&=
4u(135y^2-20yu^2+140y-84u^4+60u^2+24).
\end{align*}
Consequently,
\[
A(N^{-2}(\mathfrak{sl}_2))
\cong
\frac{\mathbb C[y,w_3,w_4,w_5]}
{(\mathcal R_0,\mathcal R_1,\mathcal R_2)}.
\]
This algebra is a free $\mathcal S$-module with basis  $1,w_4,w_5$.
\end{theorem}

\begin{proof}
Put $B^{\mathrm{crit}}=\mathbb C[u,y]$ and regard it as an
$\mathcal S$-algebra through
\[
w_3\longmapsto-8u^3+(12y+8)u.
\]
There is an $\mathcal S$-algebra isomorphism
\begin{equation}\label{eq:critical-cubic-ring-isomorphism}
\frac{\mathcal S[t]}
{(8t^3-(12y+8)t+w_3)}
\longrightarrow B^{\mathrm{crit}},
\qquad t\longmapsto u.
\end{equation}
Indeed, the relation eliminates $w_3$, and both sides are isomorphic
to $\mathbb C[t,y]$. Thus $B^{\mathrm{crit}}$ is a free
$\mathcal S$-module with basis $1,t,t^2$.

Reduction modulo the cubic relation in
\eqref{eq:critical-cubic-ring-isomorphism} gives
\begin{align}
\psi_{-2}(w_4)
&=34y^2+24y+15w_3t-60yt^2,\label{eq:critical-reduction-w4}\\
\psi_{-2}(w_5)
&=(73y+12)w_3-168y(2y+1)t+42w_3t^2.
\label{eq:critical-reduction-w5}
\end{align}
The determinant of the coefficients of $t$ and $t^2$ is
\[
630\bigl(w_3^2-16y^2(2y+1)\bigr),
\]
which is a nonzero element of the integral domain $\mathcal S$. Hence
$1,\psi_{-2}(w_4),\psi_{-2}(w_5)$ are linearly independent over
$\mathcal S$.

The homomorphism $\psi_{-2}$ is induced by the free-field characters.
Therefore $\ker\varphi_{-2}\subset\ker\psi_{-2}$.
Substitution of \eqref{eq:critical-w3}--\eqref{eq:critical-w5} into
$\mathcal R_0,\mathcal R_1,\mathcal R_2$ gives
$\psi_{-2}(\mathcal R_i)=0$ for $0\leq i\leq2$.
Thus $I_{-2}\subset\ker\psi_{-2}$. The relation
$\mathcal R_0$ expresses $w_4^2$ as an
$\mathcal S$-linear combination of $1,w_4,w_5$.
The relation $\mathcal R_1$ then expresses $w_4w_5$ in the same form,
and $\mathcal R_2$ expresses $w_5^2$ in the same form.
Therefore $P_{-2}/I_{-2}$ is generated over $\mathcal S$ by
$1,w_4,w_5$. Lemma~\ref{lem:kernel-criterion} gives
\[
I_{-2}=\ker\psi_{-2}.
\]

The polynomials
$\mathcal R_0,\mathcal R_1,\widetilde{\mathcal R}_2$ express
$w_4^2,w_4w_5,w_5^2$, successively, as $\mathcal S$-linear
combinations of $1,w_4,w_5$. Hence $P_{-2}/J_{-2}$ is generated over
$\mathcal S$ by  $1,w_4,w_5$. Since
$J_{-2}\subset\ker\varphi_{-2}\subset\ker\psi_{-2}$,
Lemma~\ref{lem:kernel-criterion} gives
\[
J_{-2}=\ker\varphi_{-2}=\ker\psi_{-2}=I_{-2}.
\]
The equality $\ker\varphi_{-2}=I_{-2}$ induces the quotient
presentation in Theorem~\ref{thm:critical-zhu-presentation}.
Moreover, Lemma~\ref{lem:kernel-criterion}, applied to $I_{-2}$,
shows that   $1,w_4,w_5$  is a
$\mathcal S$-basis of $P_{-2}/I_{-2}$.
\end{proof}

\begin{theorem}\label{thm:critical-universal-classification}
Every irreducible highest weight
$N^{-2}(\mathfrak{sl}_2)$--module is isomorphic to
$L_{-2}[x,y]$ for some $(x,y)\in\mathbb C^2$.
\end{theorem}

\begin{proof}
By Theorem~\ref{thm:critical-zhu-presentation},
$A(N^{-2}(\mathfrak{sl}_2))$ is a finitely generated
commutative $\mathbb C$-algebra. Hence its irreducible
representations are one-dimensional. The same theorem shows
that $\psi_{-2}$ induces an injective homomorphism from
$A(N^{-2}(\mathfrak{sl}_2))$ into $\mathbb C[u,y]$. We
identify $A(N^{-2}(\mathfrak{sl}_2))$ with its image under
this homomorphism.

By \eqref{eq:critical-cubic-ring-isomorphism},
$\mathbb C[u,y]$ is a free $\mathcal S$-module with basis
$1,u,u^2$. Since
$\mathcal S\subset A(N^{-2}(\mathfrak{sl}_2))
\subset\mathbb C[u,y]$, it is finitely generated as an
$A(N^{-2}(\mathfrak{sl}_2))$-module and is therefore
integral over $A(N^{-2}(\mathfrak{sl}_2))$.

Let $M$ be an irreducible highest weight
$N^{-2}(\mathfrak{sl}_2)$--module, and let
$\chi:A(N^{-2}(\mathfrak{sl}_2))\to\mathbb C$ be the
character on its top level. The lying-over argument from the
proof of Theorem~\ref{thm:zhu-parametrization} gives a
maximal ideal $\mathfrak n$ of $\mathbb C[u,y]$ whose
intersection with $A(N^{-2}(\mathfrak{sl}_2))$ is
$\ker\chi$. Therefore
$\mathfrak n=(u-u_0,y-y_0)$ for some
$(u_0,y_0)\in\mathbb C^2$, and $\chi$ is the restriction of
evaluation at $(u_0,y_0)$.

By \eqref{eq:critical-w3}--\eqref{eq:critical-w5}, this is
the character of $L_{-2}[u_0-1,y_0]$. Thus the top levels of
$M$ and $L_{-2}[u_0-1,y_0]$ are isomorphic irreducible
$A(N^{-2}(\mathfrak{sl}_2))$-modules. Zhu's correspondence
then gives
$M\cong L_{-2}[u_0-1,y_0]$.
\end{proof}

We now consider the simple quotient. Since
\[
L_{-2}(\mathfrak{sl}_2)
\cong V^{-2}(\mathfrak{sl}_2)/\langle S\rangle
\]
\cite[Section~2]{ArakawaCriticalW}, it follows that
\[
N_{-2}(\mathfrak{sl}_2)
\cong N^{-2}(\mathfrak{sl}_2)/\langle S\rangle.
\]
In particular,
$N^{-2}(\mathfrak{sl}_2)/\langle S\rangle$ is a simple vertex
algebra.
The preceding theorem gives the following classification.

\begin{corollary}\label{cor:critical-simple-classification}
Every irreducible highest weight $N_{-2}(\mathfrak{sl}_2)$--module is
isomorphic to $L_{-2}[x,0]$ for some $x\in\mathbb C$.
Their highest weights are
\[
\begin{aligned}
w_3(x,0)&=-8x(x+1)(x+2),\\
w_4(x,0)&=-120x(x+1)^2(x+2),\\
w_5(x,0)&=-48x(x+1)(x+2)(7x^2+14x+9).
\end{aligned}
\]
\end{corollary}

\section{Realization of
$N^k(\mathfrak{sl}_2)$--modules}
\label{sec:realization-simple-quotients}
Throughout this section, $k\in\mathbb C\setminus\{0,-2\}$.
Using inverse quantum Hamiltonian reduction, we realize every module
$L_k[x,y]$ inside an irreducible weight
$V^k(\mathfrak{sl}_2)$-module. We then give a criterion for such a
module to be an $N_k(\mathfrak{sl}_2)$--module.

\vskip 5mm

We use the notation $\mathbf c=\frac{2}{k}(\mu-\nu)$.
For a weight $V^k(\mathfrak{sl}_2)$-module $M$ with semisimple $h(0)$ and
$\alpha\in\mathbb C$, set
\[
  \Omega_\alpha(M)=
  \{m\in M\mid h(n)m=0\ (n>0),\ h(0)m=\alpha m\}.
\]
Following \cite{ACR,ALY}, we call $\Omega_\alpha(M)$ the
$N^k(\mathfrak{sl}_2)$--module of charge $\alpha$ associated with $M$.
If the restriction of $M$ has a decomposition
\[
M\downarrow_{\mathcal H^k\otimes N^k(\mathfrak{sl}_2)}
\cong\bigoplus_{\alpha}
\mathcal F^k_\alpha\otimes C_\alpha,
\]
then $C_\alpha\cong\Omega_\alpha(M)$ as
$N^k(\mathfrak{sl}_2)$--modules.

Let $\mathcal F^k_\alpha$ denote the irreducible
$\mathcal H^k$-module generated by a vector on which $h(0)$ acts as
$\alpha$ and $h(n)$ acts trivially for $n>0$.

We first fix notation for the four types of irreducible weight
$V^k(\mathfrak{sl}_2)$-modules used below. If $U$ is an irreducible weight
$\mathfrak{sl}_2$-module, extend it to a
$\mathfrak{sl}_2[t]\oplus\mathbb CK$-module by letting
$\mathfrak{sl}_2\otimes t\mathbb C[t]$ act trivially and $K$ act as $k$, and
form
\[
  \mathbb V^k(U)=
  U(\widehat{\mathfrak{sl}}_2)
  \otimes_{U(\mathfrak{sl}_2[t]\oplus\mathbb CK)}U.
\]
The module $\mathbb V^k(U)$ has a unique irreducible quotient. We
identify $U$ with its image in the top space of this quotient.

Let $\Omega_{\mathfrak{sl}_2}=\frac12h^2+ef+fe$ be the quadratic
Casimir operator. We use the following notation.
\begin{itemize}[leftmargin=2.5em]
\item For $\lambda\in\mathbb C/2\mathbb Z$,
$E_{\lambda;\Delta}$ denotes an irreducible relaxed highest-weight
$V^k(\mathfrak{sl}_2)$-module whose top space is an irreducible dense
weight $\mathfrak{sl}_2$-module with support
$\lambda+2\mathbb Z$ and one-dimensional weight spaces. The operator
$\Omega_{\mathfrak{sl}_2}$ acts on the top space as
$2(k+2)\Delta$.

\item For $\lambda\in\mathbb C\setminus\mathbb Z_{\geq0}$,
$D^+_{\lambda}$ denotes the irreducible highest-weight
$V^k(\mathfrak{sl}_2)$-module whose top space is the irreducible
highest-weight $\mathfrak{sl}_2$-module of highest weight $\lambda$.

\item For $\lambda\in\mathbb C\setminus\mathbb Z_{\geq0}$,
$D^-_{\lambda}$ denotes the conjugate highest-weight module obtained
from $D^+_{\lambda}$ by the involution
$e(n)\longleftrightarrow f(n)$, $h(n)\mapsto-h(n)$.
Its top space is the irreducible lowest-weight
$\mathfrak{sl}_2$-module of lowest weight $-\lambda$.

\item For $d\in\mathbb Z_{\ge1}$, set
$L_d=L_k((d-1)\omega_1)$.
Its top space is the $d$-dimensional irreducible $\mathfrak{sl}_2$-module of
highest weight $d-1$.
\end{itemize}

The following irreducibility statement is standard. It may be viewed
as an application of quantum Galois theory to the $h(0)$-weight
decomposition of an irreducible module. Similar arguments are used in
\cite{ACR}; see also \cite{A-2005,A2007} for earlier applications in
special situations.

\begin{lemma}\label{lem:omega-irred}
Let $M$ be an irreducible weight $V^k(\mathfrak{sl}_2)$-module such that
$h(0)$ acts semisimply and
$\operatorname{Supp}_h M\subset\lambda+2\mathbb Z$.
For $\ell\in\mathbb Z$, set
$M^{(\ell)}=\{m\in M\mid h(0)m=(\lambda+2\ell)m\}$
and
\[
\Omega_{\lambda+2\ell}(M)
 =\{m\in M^{(\ell)}\mid h(n)m=0\ (n>0)\}.
\]
For $s\in\mathbb Z$, put
$V^k(\mathfrak{sl}_2)^{(s)}
=\{a\in V^k(\mathfrak{sl}_2)\mid h(0)a=2s\,a\}$.
\begin{enumerate}[label=\textup{(\roman*)}]
\item Every nonzero $M^{(\ell)}$ is an irreducible module for the
subalgebra $V^k(\mathfrak{sl}_2)^{(0)}$.

\item Assume, in addition, that every nonzero $h(0)$-weight space has
the decomposition
\[
  M^{(\ell)}
  \cong\mathcal F^k_{\lambda+2\ell}\otimes
  \Omega_{\lambda+2\ell}(M)
\]
as an $\mathcal H^k\otimes N^k(\mathfrak{sl}_2)$-module. Then every
nonzero $\Omega_{\lambda+2\ell}(M)$ is an irreducible
$N^k(\mathfrak{sl}_2)$-module.
\end{enumerate}
\end{lemma}

\begin{proof}
We recall the standard argument.
Here $A\cdot X$ denotes the linear span of all vectors $a_nx$ with
$a\in A$, $x\in X$, and $n\in\mathbb Z$. Then
$V^k(\mathfrak{sl}_2)^{(s)}\cdot M^{(\ell)}
\subseteq M^{(\ell+s)}$.
Let $0\neq w\in M^{(\ell)}$.  Since $M$ is irreducible,
$V^k(\mathfrak{sl}_2)\cdot w=M$. Since $M$ is
$\mathbb Z$--graded with respect to $h(0)$, taking the
$(\ell+s)$-graded component gives
$V^k(\mathfrak{sl}_2)^{(s)}\cdot w=M^{(\ell+s)}$.

Now let $0\neq U\subseteq M^{(\ell)}$ be a
$V^k(\mathfrak{sl}_2)^{(0)}$-submodule and choose $0\neq u\in U$.
Again, irreducibility of $M$ gives
$V^k(\mathfrak{sl}_2)\cdot u=M$.  Taking the $\ell$-graded component
shows that
$M^{(\ell)}=V^k(\mathfrak{sl}_2)^{(0)}\cdot u\subseteq U$.
Hence $U=M^{(\ell)}$, which proves \textup{(i)}.

Assume the additional hypothesis in \textup{(ii)}. We have
$V^k(\mathfrak{sl}_2)^{(0)}
\cong\mathcal H^k\otimes N^k(\mathfrak{sl}_2)$.
By \textup{(i)}, $M^{(\ell)}$ is irreducible as an
$\mathcal H^k\otimes N^k(\mathfrak{sl}_2)$-module. If
$0\neq U\subsetneq\Omega_{\lambda+2\ell}(M)$ were a proper
$N^k(\mathfrak{sl}_2)$-submodule, then
$\mathcal F^k_{\lambda+2\ell}\otimes U$
would be a proper nonzero
$\mathcal H^k\otimes N^k(\mathfrak{sl}_2)$-submodule of
$M^{(\ell)}$, a contradiction.
Therefore $\Omega_{\lambda+2\ell}(M)$ is irreducible as an
$N^k(\mathfrak{sl}_2)$-module.
\end{proof}

We now determine the $N^k(\mathfrak{sl}_2)$--modules in the
decompositions of the four types of
$V^k(\mathfrak{sl}_2)$--modules introduced above. Such decompositions
were studied in a more general setting in \cite{CKLR}.

\begin{theorem}[Decompositions for $V^k(\mathfrak{sl}_2)$-modules]
\label{thm:universal-affine-decompositions}
Let $k\in\mathbb C\setminus\{0,-2\}$.  The following decompositions hold for
the four types of irreducible weight $V^k(\mathfrak{sl}_2)$-modules
introduced above.
\begin{enumerate}[label=\textup{(\roman*)}]
\item For an irreducible relaxed highest-weight module $E_{\lambda;\Delta}$,
\[
E_{\lambda;\Delta}\downarrow_{\mathcal H^k\otimes N^k(\mathfrak{sl}_2)}
\cong
\bigoplus_{\nu\in\lambda+2\mathbb Z}
\mathcal F^k_\nu\otimes C^E_{\nu;\Delta}.
\]
The modules $C^E_{\nu;\Delta}$ are irreducible
$N^k(\mathfrak{sl}_2)$--modules, and
\[
C^E_{\nu;\Delta}
\cong
L_k\left[\frac{k+\nu}{2},\Delta-\frac{k}{4}\right].
\]

\item Let $D^+_\lambda$ be an irreducible highest weight module and put
\[
\Delta_\lambda=\frac{\lambda(\lambda+2)}{4(k+2)}.
\]
Then
\[
D^+_\lambda\downarrow_{\mathcal H^k\otimes N^k(\mathfrak{sl}_2)}
\cong
\bigoplus_{\nu\in\lambda+2\mathbb Z}
\mathcal F^k_\nu\otimes C^{D,+}_{\nu;\lambda}.
\]
Every nonzero module $C^{D,+}_{\nu;\lambda}$ is an
irreducible $N^k(\mathfrak{sl}_2)$--module.
For the weights in the top space,
\[
\nu_i^+=\lambda-2i,
\qquad i\in\mathbb Z_{\geq0},
\]
one has
\[
C^{D,+}_{\nu_i^+;\lambda}
\cong
L_k\left[\frac{k+\nu_i^+}{2},
\Delta_\lambda-\frac{k}{4}\right].
\]

\item Let $D^-_\lambda$ be an irreducible conjugate highest-weight module.
Then
\[
D^-_\lambda\downarrow_{\mathcal H^k\otimes N^k(\mathfrak{sl}_2)}
\cong
\bigoplus_{\nu\in-\lambda+2\mathbb Z}
\mathcal F^k_\nu\otimes C^{D,-}_{\nu;\lambda}.
\]
Every nonzero module $C^{D,-}_{\nu;\lambda}$ is an
irreducible $N^k(\mathfrak{sl}_2)$--module.
For the weights in the top space,
\[
\nu_i^-=-\lambda+2i,
\qquad i\in\mathbb Z_{\geq0},
\]
one has
\[
C^{D,-}_{\nu_i^-;\lambda}
\cong
L_k\left[\frac{k+\nu_i^-}{2},
\Delta_\lambda-\frac{k}{4}\right].
\]

\item Let $L_d$ be the irreducible ordinary module whose top space is the
$d$-dimensional irreducible $\mathfrak{sl}_2$-module, and put
\[
\Delta_d=\frac{d^2-1}{4(k+2)}.
\]
Then
\[
L_d\downarrow_{\mathcal H^k\otimes N^k(\mathfrak{sl}_2)}
\cong
\bigoplus_{\nu\in d-1+2\mathbb Z}
\mathcal F^k_\nu\otimes C^L_{\nu;d}.
\]
Every nonzero module $C^L_{\nu;d}$ is an irreducible
$N^k(\mathfrak{sl}_2)$--module. For the weights in the top space,
\[
\nu_i=d-1-2i,
\qquad 0\leq i\leq d-1,
\]
one has
\[
C^L_{\nu_i;d}
\cong
L_k\left[\frac{k+\nu_i}{2},
\Delta_d-\frac{k}{4}\right].
\]
\end{enumerate}

Every irreducible highest weight $N^k(\mathfrak{sl}_2)$-module is
isomorphic to one of the $N^k(\mathfrak{sl}_2)$--modules listed in
\textup{(i)--(iv)}. More precisely, for every $(x,y)\in\mathbb C^2$, there
exist an irreducible weight $V^k(\mathfrak{sl}_2)$-module $M$ of one of the
four types above, a nonzero vector $w$ in its top space, and
$\nu\in\mathbb C$ such that $h(0)w=\nu w$ and
\[
N^k(\mathfrak{sl}_2)w
=\Omega_{\nu}(M)
\cong L_k[x,y].
\]
\end{theorem}

\begin{proof}

The decompositions in \textup{(i)--(iv)} are decompositions with
respect to the Heisenberg Fock modules. For each of the four types of
modules, every nonzero $N^k(\mathfrak{sl}_2)$-module occurring in the
corresponding decomposition satisfies the hypothesis of
Lemma~\ref{lem:omega-irred}\textup{(ii)} and is therefore irreducible.

The inclusion $N^k(\mathfrak{sl}_2)\subset V^k(\mathfrak{sl}_2)$
induces a homomorphism
\[
 A(N^k(\mathfrak{sl}_2))
 \longrightarrow A(V^k(\mathfrak{sl}_2))
 \cong U(\mathfrak{sl}_2).
\]
Its image is contained in
\[
 U(\mathfrak{sl}_2)^h
 =\mathbb C[h,\Omega_{\mathfrak{sl}_2}].
\]
Let $w$ belong to the top space of a weight
$V^k(\mathfrak{sl}_2)$-module, and assume that
\[
 h(0)w=\nu w,\qquad
 \Omega_{\mathfrak{sl}_2}w=2(k+2)\Delta w.
\]
The inverse quantum Hamiltonian reduction formulas of
Section~\ref{sec:vir-heis-realization}, together with the character
calculation in Section~\ref{sec:evaluation}, give the
$A(N^k(\mathfrak{sl}_2))$-character
\[
 \chi^k_{\frac{k+\nu}{2},\,\Delta-\frac{k}{4}}
\]
on $\mathbb Cw$.

For $E_{\lambda;\Delta}$, the Casimir operator acts on the top space as
$2(k+2)\Delta$. Therefore, for
$\nu\in\lambda+2\mathbb Z$, the preceding formula for the
$A(N^k(\mathfrak{sl}_2))$-character and the
irreducibility of $C^E_{\nu;\Delta}$ give
\[
C^E_{\nu;\Delta}
\cong
L_k\left[\frac{k+\nu}{2},\Delta-\frac{k}{4}\right].
\]

On the top spaces of $D^+_\lambda$ and $D^-_\lambda$, the Casimir
operator acts as
\[
\frac{\lambda(\lambda+2)}{2}
=2(k+2)\Delta_\lambda.
\]
Applying the formula for the $A(N^k(\mathfrak{sl}_2))$-character to
the top-space weights
$\nu_i^+=\lambda-2i$ and $\nu_i^-=-\lambda+2i$ gives
\[
C^{D,+}_{\nu_i^+;\lambda}
\cong
L_k\left[\frac{k+\nu_i^+}{2},
\Delta_\lambda-\frac{k}{4}\right].
\]
\[
C^{D,-}_{\nu_i^-;\lambda}
\cong
L_k\left[\frac{k+\nu_i^-}{2},
\Delta_\lambda-\frac{k}{4}\right].
\]

The Casimir operator acts on the top space of $L_d$ as
\[
\frac{d^2-1}{2}=2(k+2)\Delta_d.
\]
Applying the formula for the $A(N^k(\mathfrak{sl}_2))$-character to
$\nu_i=d-1-2i$ gives
\[
C^L_{\nu_i;d}
\cong
L_k\left[
\frac{k+\nu_i}{2},
\Delta_d-\frac{k}{4}
\right].
\]

For the exhaustion statement, put
\[
A_n(x,y)=
(k+2)\left(y+\frac{k}{4}\right)-
\left(x-\frac{k}{2}+n\right)
\left(x-\frac{k}{2}+n-1\right).
\]
Using the embedding \eqref{inverse-embed}, consider the
$V^{\Vir}(c_k,0)\otimes\Pi(0)$-module
\[
\mathcal M[x,y]
=L^{\Vir}(c_k,y)\otimes\Pi\cdot e^{-\mu+x\mathbf c}.
\]
Let $v_y$ be a highest weight vector of $L^{\Vir}(c_k,y)$. The top
space of $\mathcal M[x,y]$ is spanned by
\[
u_n(x,y)=v_y\otimes e^{-\mu+(x+n)\mathbf c},\qquad n\in\mathbb Z.
\]
The $\mathfrak{sl}_2$-action on
$\mathcal M[x,y]_{\mathrm{top}}$ is given by

\begin{equation}
h(0)u_n=(-k+2x+2n)u_n,\qquad
e(0)u_n=u_{n+1},\qquad
f(0)u_n=A_n(x,y)u_{n-1}.
\label{top-space}
\end{equation}

The cyclic module $N^k(\mathfrak{sl}_2)u_n(x,y)$ is a highest weight
$N^k(\mathfrak{sl}_2)$--module with character $\chi^k_{x+n,y}$. In
particular, $u_0(x,y)$ generates a highest weight module
$\widetilde L_k[x,y]$ with character $\chi^k_{x,y}$.

The arguments of \cite[Sections~3.3 and~3.4]{AKR24} apply to the
$\mathfrak{sl}_2$ case; see also \cite[Section~1.2]{AKR24}.

If $\mathcal M[x,y]_{\mathrm{top}}$ is
irreducible, then $\mathcal M[x,y]$ is irreducible and is of
type~\textup{(i)}. If $\mathcal M[x,y]_{\mathrm{top}}$ is reducible,
then $\mathcal M[x,y]$ contains an irreducible submodule of
type~\textup{(iii)}, or has irreducible subquotients of
types~\textup{(ii)} and~\textup{(iv)}. Therefore, every irreducible
$V^k(\mathfrak{sl}_2)$--module of types~\textup{(i)--(iv)} appears as
$\mathcal M[x,y]$, an irreducible submodule of $\mathcal M[x,y]$, or
an irreducible subquotient of $\mathcal M[x,y]$.

Fix $(x,y)\in\mathbb C^2$ and set
\[
Z(x,y)=\{n\in\mathbb Z\mid A_n(x,y)=0\}.
\]
Since $A_n(x,y)$ is quadratic in $n$, the set $Z(x,y)$ has at most two
elements. We consider the following cases.
\begin{itemize}
\item If $Z(x,y)=\varnothing$, then
$L_k[x,y]=\widetilde L_k[x,y]$ is realized inside an irreducible
$V^k(\mathfrak{sl}_2)$--module of type~\textup{(i)}.

\item If
$\varnothing\neq Z(x,y)\subset\mathbb Z_{>0}$, then $L_k[x,y]$ is
realized inside an irreducible $V^k(\mathfrak{sl}_2)$--module of
type~\textup{(ii)}.

\item If
$\varnothing\neq Z(x,y)\subset\mathbb Z_{\leq0}$, then $L_k[x,y]$ is
realized inside an irreducible $V^k(\mathfrak{sl}_2)$--module of
type~\textup{(iii)}.

\item The remaining case is
$Z(x,y)=\{m_1,m_2\}$, where $m_1\leq0<m_2$.
Put $d=m_2-m_1$. Solving
$A_{m_1}(x,y)=A_{m_2}(x,y)=0$ gives
\[
x=\frac{k+1-m_1-m_2}{2},
\qquad
y=\frac{d^2-1}{4(k+2)}-\frac{k}{4}.
\]
In this case, the top space has a $d$-dimensional irreducible
$\mathfrak{sl}_2$-subquotient containing the class of $u_0(x,y)$.
Taking $i=m_2-1$ in \textup{(iv)} gives
\[
  \nu_i=1-m_1-m_2,
  \qquad
  \frac{k+\nu_i}{2}=x.
\]
Hence $C^L_{\nu_i;d}\cong L_k[x,y]$.
\end{itemize}

This proves the final assertion of the theorem.
\end{proof}

We shall use the following no-zero-divisors property of intertwining
operators.
\begin{lemma}[{\cite{DL}}]\label{lem:nonzero-divisor}
Let $V$ be a vertex algebra, let $W_1$ and $W_2$ be simple
$V$-modules, and let $W_3$ be a $V$-module. If
$\mathcal Y\in I\binom{W_3}{W_1\;W_2}$ is a nonzero intertwining
operator, then $\mathcal Y(w_1,z)h\neq0$
for every $0\neq w_1\in W_1$ and $0\neq h\in W_2$.
\end{lemma}

Recall that $V^k(\mathfrak{sl}_2)$ is non-simple if and only if the
level $k$ is admissible or critical. We now assume that $k$ is
admissible. Let $\mathcal I_k$ be the maximal ideal of
$V^k(\mathfrak{sl}_2)$.

The following result is \cite[Theorem~5.3]{AKMPP17}.

\begin{proposition}\label{prop:Ik-irreducible}
Assume that the level  $k$ is admissible. Then the maximal ideal $\mathcal I_k$ of
$V^k(\mathfrak{sl}_2)$ is nonzero and simple.
\end{proposition}

Set $\mathcal J_k=\mathcal I_k\cap N^k(\mathfrak{sl}_2)$.
We shall also use that $\mathcal J_k\neq0$. Indeed, write
$k=-2+\frac{p}{q}$, where $p,q\in\mathbb N$ and $(p,q)=1$,
and let $v_\lambda$ be the singular vector generating $\mathcal I_k$. By
\cite[Theorem~5.3(1)]{AKMPP17}, its
$\mathfrak{sl}_2$-highest weight is $2(p-1)$. The homogeneous subspace
containing $v_\lambda$ is finite-dimensional and semisimple as an
$\mathfrak{sl}_2$-module. Consequently, $v_\lambda$ generates the
irreducible $\mathfrak{sl}_2$-module of highest weight $2(p-1)$, and
$0\neq f(0)^{p-1}v_\lambda\in\mathcal I_k$.
This vector has $h(0)$-weight zero. Since $v_\lambda$ is singular, the
commutation relations also give
$h(n)f(0)^{p-1}v_\lambda=0$ for $n>0$.
By the characterization
\[
  N^k(\mathfrak{sl}_2)
  =
  \{a\in V^k(\mathfrak{sl}_2)\mid
  h(0)a=0,\ h(n)a=0\ (n>0)\},
\]
this vector belongs to $\mathcal J_k$. Therefore $\mathcal J_k\neq0$.
Moreover, $\mathcal J_k$ is the kernel of the restriction to
$N^k(\mathfrak{sl}_2)$ of the quotient homomorphism
$V^k(\mathfrak{sl}_2)\to L_k(\mathfrak{sl}_2)$. Hence
\[
  N_k(\mathfrak{sl}_2)=N^k(\mathfrak{sl}_2)/\mathcal J_k.
\]

\par\medskip
\begin{proposition}\label{prop:factorization-criterion}
Assume that $k$ is admissible. Let $M$ be an irreducible weight
$V^k(\mathfrak{sl}_2)$-module, let
$0\neq w\in M$ be an $\mathcal H^k$-highest vector, and put
$L=N^k(\mathfrak{sl}_2)w$.
Then
\[
  \mathcal J_kL=0 \quad\Longleftrightarrow\quad \mathcal I_kM=0.
\]
Consequently, $L$ is an $N_k(\mathfrak{sl}_2)$-module if and only if
$M$ is an $L_k(\mathfrak{sl}_2)$-module.
\end{proposition}

\begin{proof}
Since $\mathcal J_k\subset\mathcal I_k$, the implication
$\mathcal I_kM=0\Longrightarrow\mathcal J_kL=0$
is immediate.

Conversely, suppose that $\mathcal I_kM\neq0$. The restriction of the module
vertex operator to $\mathcal I_k\otimes M$,
\[
  \mathcal Y(\,\cdot\,,z)
  =Y_M(\,\cdot\,,z)\big|_{\mathcal I_k\otimes M},
\]
is an intertwining operator of type
\[
  \binom{M}{\mathcal I_k\;M}.
\]
The assumption $\mathcal I_kM\neq0$ means precisely that this intertwining
operator is nonzero. By Proposition~\ref{prop:Ik-irreducible},
$\mathcal I_k$ is a
simple $V^k(\mathfrak{sl}_2)$-module, and $M$ is irreducible by
assumption. Choose $0\neq u\in\mathcal J_k$. By
Lemma~\ref{lem:nonzero-divisor},
$Y_M(u,z)w=\mathcal Y(u,z)w\neq0$.
Hence $u_nw\neq0$ for some $n\in\mathbb Z$. Since
$u\in\mathcal J_k$ and $w\in L$, it follows that
$\mathcal J_kL\neq0$. Thus
$\mathcal J_kL=0\Longrightarrow\mathcal I_kM=0$,
and the proof is complete.
\end{proof}

\section{Classification of irreducible highest weight $N_k(\mathfrak{sl}_2)$--modules at non-integral admissible levels}
\label{sec:admissible}
\label{sec:admissible-simple-quotients}

In this section we classify the irreducible highest weight
$N_k(\mathfrak{sl}_2)$--modules at non-integral admissible levels. We first
describe the continuous families obtained from inverse quantum Hamiltonian
reduction. We then determine the additional modules obtained from ordinary
$L_k(\mathfrak{sl}_2)$--modules.
\vskip 5mm

Assume $k=-2+\frac{u}{v}$, where $u,v\ge2$ and $(u,v)=1$.
Put
\[
  h^{u,v}_{r,s}
  =\frac{(vr-us)^2-(u-v)^2}{4uv},
  \qquad
  1\le r\le u-1,\quad 1\le s\le v-1,
\]
and
\[
  \mathcal H_{u,v}=\{h^{u,v}_{r,s}\mid
  1\le r\le u-1,\ 1\le s\le v-1\}.
\]
We also put $\lambda_{r,s}=r-1-\frac{u}{v}s$.
For $1\le r\le u-1$ and $0\le i\le r-1$, set
\[
  \Delta_r=\frac{r^2-1}{4(k+2)}=\frac{v(r^2-1)}{4u},
  \qquad
  x_{r,i}=\frac{k-r+1+2i}{2},
  \qquad
  y_r=\Delta_r-\frac{k}{4}.
\]

The ordinary module $L_k((r-1)\omega_1)$ gives the
$N_k(\mathfrak{sl}_2)$--modules
\[
  L_k[x_{r,i},y_r],\qquad 0\le i\le r-1.
\]
Here the index $i$ orders the top-space weights from lowest to highest:
$\nu_i^{\mathrm{ord}}=1-r+2i$, with
$x_{r,i}=(k+\nu_i^{\mathrm{ord}})/2$.
The two extremal modules already belong to the continuous families. More
precisely,
\begin{equation}\label{eq:ordinary-boundary-identifications}
  L_k[x_{r,0},y_r]
  \cong
  L_k\left[x_{r,0}+\frac{k}{2},h^{u,v}_{r,1}\right],
  \qquad
  L_k[x_{r,r-1},y_r]
  \cong
  L_k\left[x_{r,r-1}-\frac{k}{2},h^{u,v}_{r,1}\right].
\end{equation}
Evaluation of the character values at the parameters in
\eqref{eq:ordinary-boundary-identifications} shows that the two
$A(N_k(\mathfrak{sl}_2))$-characters on the two sides of each
isomorphism in \eqref{eq:ordinary-boundary-identifications} are equal.
The top levels are therefore isomorphic irreducible
$A(N_k(\mathfrak{sl}_2))$-modules. Since all four
$N_k(\mathfrak{sl}_2)$--modules are irreducible, Zhu's correspondence
gives the two isomorphisms in
\eqref{eq:ordinary-boundary-identifications}. The remaining ordinary
modules form the family
\begin{equation}\label{eq:Oreduced-explicit}
  \mathcal O^{\mathrm{red}}_{u,v}
  =
  \left\{
  L_k\left[
  \frac{k-r+1+2i}{2},
  \frac{v(r^2-1)}{4u}-\frac{k}{4}
  \right]
  \ \middle|\
  3\le r\le u-1,
  \ 1\le i\le r-2
  \right\}.
\end{equation}

\begin{lemma}\label{lem:ordinary-characters-distinct}
The $A(N_k(\mathfrak{sl}_2))$-characters of the modules in
$\mathcal O^{\mathrm{red}}_{u,v}$ are pairwise distinct.  In particular,
these modules are pairwise non-isomorphic.
\end{lemma}

\begin{proof}
Set $m_{r,i}=r-1-2i=-\nu_i^{\mathrm{ord}}$ and
$\chi_{r,i}=\chi^k_{x_{r,i},y_r}$.
A direct calculation gives
\begin{equation}\label{eq:ordinary-hq}
\begin{aligned}
  h_{r,i}:=\chi_{r,i}([L])
  &=\Delta_r-\frac{m_{r,i}^2}{4k},\\
  q_{r,i}:=\chi_{r,i}([W])
  &=m_{r,i}\bigl(m_{r,i}^2-k^2
  +6k(k+2)h_{r,i}\bigr).
\end{aligned}
\end{equation}
Assume that $\chi_{r,i}=\chi_{s,j}$, and write
$m=m_{r,i}$, $n=m_{s,j}$, and
$\eta=\chi_{r,i}([L])=\chi_{s,j}([L])$.
If $m=n$, the first identity in \eqref{eq:ordinary-hq} gives
$\Delta_r=\Delta_s$. Hence $r=s$, and therefore $i=j$.

Suppose that $m\neq n$. The equality
$\chi_{r,i}([W])=\chi_{s,j}([W])$ implies
\begin{equation}\label{eq:ordinary-W-difference}
  m^2+mn+n^2-k^2+6k(k+2)\eta=0.
\end{equation}
The equality
$\chi_{r,i}([W^4])=\chi_{s,j}([W^4])$ gives
\begin{equation} \label{eq:ordinary-W4-Z4-difference}
(m^2-n^2)\bigl(m^2+n^2-k^2+4k(k+2)\eta\bigr)=0.
\end{equation}
At either exceptional level, the same identity follows from
$\chi_{r,i}([Z^4])=\chi_{s,j}([Z^4])$.

If $m^2\neq n^2$, the last equation implies
\begin{equation}\label{eq:ordinary-W4-difference}
  m^2+n^2-k^2+4k(k+2)\eta=0.
\end{equation}
Equations \eqref{eq:ordinary-W-difference} and
\eqref{eq:ordinary-W4-difference} imply $(m-n)^2=k^2$,
which is impossible since $m-n\in\mathbb Z$ and the admissible level
$k=-2+u/v$ is not an integer. Hence $m^2=n^2$. Since $m\neq n$, it
remains to exclude the case $n=-m$ with $m\neq0$.

Define
\[
\widehat f_5(k,m,\eta)
=f_5\left(
k,\frac{k-m}{2},
\eta-\frac{k-m}{2}+\frac{(k-m)^2}{4k}
\right),
\]
and define $\widehat g_5(k,m,\eta)$ by the same substitution in $g_5$.
The equality of the character values on $[W^5]$ gives
$\widehat f_5(k,m,\eta)=\widehat f_5(k,-m,\eta)$.
At either exceptional level, equality of the values on $[Z^5]$ gives
the analogous identity for $\widehat g_5$. Direct substitution shows that
both transformed polynomials are odd in $m$. Therefore
$\widehat f_5(k,m,\eta)=0$ in the non-exceptional case and
$\widehat g_5(k,m,\eta)=0$ at either exceptional level.

Equality of the values on $[W]$ gives
$\eta=(k^2-m^2)/(6k(k+2))$.
At this value of $\eta$, direct substitution gives
\[
\widehat f_5(k,m,\eta)
=-\frac{64k+107}{9}\,
m(m^2-k^2)(4m^2-k^2)
\]
and
\[
\widehat g_5(k,m,\eta)
=24k^2(2k+3)(3k+4)\,
m(m^2-k^2)(4m^2-k^2).
\]
The first scalar coefficient is nonzero at every non-exceptional level
under consideration. The second scalar coefficient is nonzero at both
exceptional levels. Since $m\neq0$, the corresponding vanishing equation
therefore implies $(m^2-k^2)(4m^2-k^2)=0$.
Since $m\in\mathbb Z$, this would imply that $k$ is an integer,
contrary to the assumption on the admissible level. Thus $n=-m$ is
impossible. This completes the proof.
\end{proof}

Here the set is empty if $u=2$ or $u=3$.  In general,
\begin{equation}\label{eq:Oreduced-cardinality}
  |\mathcal O^{\mathrm{red}}_{u,v}|
  =\sum_{r=3}^{u-1}(r-2)
  =\frac{(u-2)(u-3)}{2}.
\end{equation}

\begin{theorem}\label{thm:admissible-classification}
Assume $k=-2+\frac{u}{v}$, where $u,v\ge2$ and $(u,v)=1$.
The irreducible highest weight $N_k(\mathfrak{sl}_2)$--modules are precisely the modules in
\[
  \bigcup_{h\in\mathcal H_{u,v}}
  \mathcal C_h
  \ \cup\
  \mathcal O^{\mathrm{red}}_{u,v},
\]
where
$\mathcal C_h=\{L_k[x,h]\mid x\in\mathbb C\}$,
and $\mathcal O^{\mathrm{red}}_{u,v}$ is given by
\eqref{eq:Oreduced-explicit}.
\end{theorem}

\begin{proof}
By Theorem~\ref{thm:universal-affine-decompositions}, every irreducible
highest weight $N^k(\mathfrak{sl}_2)$--module occurs in the restriction of
an irreducible weight $V^k(\mathfrak{sl}_2)$-module of one of the four types
listed there. By Proposition~\ref{prop:factorization-criterion}, such an
$N^k(\mathfrak{sl}_2)$--module is an
$N_k(\mathfrak{sl}_2)$--module if and only if
the  $V^k(\mathfrak{sl}_2)$--module in which it occurs is an
$L_k(\mathfrak{sl}_2)$--module.
Therefore it remains to apply the classification of irreducible relaxed
highest weight $L_k(\mathfrak{sl}_2)$-modules at admissible level
\cite[Theorem~4.0.11]{AM95}; see also \cite[Theorem~2.3]{ACR}.

Indeed, let $X$ be an irreducible highest weight
$N_k(\mathfrak{sl}_2)$--module. By
Theorem~\ref{thm:universal-affine-decompositions}, there are an
irreducible weight $V^k(\mathfrak{sl}_2)$-module $M$ and a vector $w$
in its top space such that
\[
X\cong N^k(\mathfrak{sl}_2)w.
\]
Since $\mathcal J_kX=0$, the preceding isomorphism gives
$\mathcal J_kN^k(\mathfrak{sl}_2)w=0$. Proposition~
\ref{prop:factorization-criterion} therefore implies that
$\mathcal I_kM=0$. Hence $M$ is an irreducible
$L_k(\mathfrak{sl}_2)$-module.

Fix $1\le r\le u-1$ and $1\le s\le v-1$. The typical relaxed modules
$E_{\lambda;\Delta^{\mathrm{aff}}_{r,s}}$ occur for
\[
  \lambda\not\equiv\lambda_{r,s},\lambda_{u-r,v-s}
  \pmod{2\mathbb Z}.
\]
By Theorem~\ref{thm:universal-affine-decompositions}, they give the
modules
\[
L_k\left[\frac{k+\nu}{2},h^{u,v}_{r,s}\right],
\qquad \nu\in\lambda+2\mathbb Z.
\]
Thus, the typical modules give all congruence classes except
$\lambda_{r,s}+2\mathbb Z$ and
$\lambda_{u-r,v-s}+2\mathbb Z$. Using
\[
\lambda_{u-r,v-s}=-\lambda_{r,s}-2,
\qquad
h^{u,v}_{u-r,v-s}=h^{u,v}_{r,s},
\]
the top spaces of $D^+_{r,s}$ and $D^-_{u-r,v-s}$ together give all
weights in $\lambda_{r,s}+2\mathbb Z$, whereas the top spaces of
$D^+_{u-r,v-s}$ and $D^-_{r,s}$ give all weights in
$\lambda_{u-r,v-s}+2\mathbb Z$. Hence the resulting
$N_k(\mathfrak{sl}_2)$--modules are exactly
\[
\mathcal C_{h^{u,v}_{r,s}}
=\{L_k[x,h^{u,v}_{r,s}]\mid x\in\mathbb C\}.
\]
Taking all $(r,s)$ gives the union of continuous families indexed by
$\mathcal H_{u,v}$.

The remaining irreducible relaxed highest weight modules are the ordinary
modules $L_k((r-1)\omega_1)$, $1\le r\le u-1$. Their top spaces have
weights $\nu_i^{\mathrm{ord}}=1-r+2i$, where $0\le i\le r-1$,
when ordered from lowest to highest. Since
$x_{r,i}=(k+\nu_i^{\mathrm{ord}})/2$, they give the modules
$L_k[x_{r,i},y_r]$. By
\eqref{eq:ordinary-boundary-identifications}, the two extremal cases
$i=0$ and $i=r-1$ already lie in the continuous families. The interior
weights give precisely $\mathcal O^{\mathrm{red}}_{u,v}$.

Thus every module in
$\bigcup_{h\in\mathcal H_{u,v}}\mathcal C_h
\cup\mathcal O^{\mathrm{red}}_{u,v}$
is an $N_k(\mathfrak{sl}_2)$--module. The classification of relaxed
highest weight $L_k(\mathfrak{sl}_2)$-modules shows that no other
irreducible highest weight $N_k(\mathfrak{sl}_2)$--modules occur.
\end{proof}

\begin{remark}
Every isomorphism class of irreducible highest weight
$N_k(\mathfrak{sl}_2)$--modules occurs in the union in
Theorem~\ref{thm:admissible-classification}. The parametrization of the
finite family $\mathcal O^{\mathrm{red}}_{u,v}$ is injective, whereas
the parametrizations of the continuous families need not be injective.
\end{remark}

\begin{remark}
The classification in
Theorem~\ref{thm:admissible-classification} agrees with the
Heisenberg-coset decompositions in \cite{ACR}. In the notation of that
paper, the irreducible highest weight
$N_k(\mathfrak{sl}_2)$-modules are
\[
\begin{aligned}
&C^E_{\mu;r,s},
&&1\leq r\leq u-1,\quad 1\leq s\leq v-1,\quad
\mu\not\equiv\lambda_{r,s},\lambda_{u-r,v-s}
\pmod{2\mathbb Z},\\
&C^D_{\mu;r,s},
&&1\leq r\leq u-1,\quad 1\leq s\leq v-1,\quad
\mu\in\lambda_{r,s}+2\mathbb Z,\\
&C^L_{\mu;r},
&&1\leq r\leq u-1,\quad
\mu\in\lambda_{r,0}+2\mathbb Z.
\end{aligned}
\]
The modules $C^E_{\mu;r,s}$ arise from typical relaxed highest-weight
$L_k(\mathfrak{sl}_2)$-modules, the modules $C^D_{\mu;r,s}$ from
highest- and lowest-weight modules, and the modules $C^L_{\mu;r}$ from
ordinary modules. The first family is typical, while the other two
families are atypical.

The corresponding $L_k(\mathfrak{sl}_2)$--modules are classified in
\cite[Theorem~4.0.11]{AM95}, and their
$\mathcal H^k\otimes N_k(\mathfrak{sl}_2)$-decompositions are given in
\cite[Propositions~3.1 and~3.3]{ACR}. The three families above were
obtained in \cite{ACR} from these decompositions.

Our Theorem~\ref{thm:admissible-classification} proves that the above modules exhaust all
irreducible highest weight $N_k(\mathfrak{sl}_2)$-modules.
\end{remark}
\section{Examples}

\label{sec:examples}
This section illustrates the classification at several special levels and
compares it with previously known realizations.

\subsection{The singlet levels $k=-\frac12$ and $k=-\frac43$}
\label{subsec:singlet-levels}

Theorem~\ref{thm:admissible-classification} also recovers the known
classifications at the two singlet levels.

The levels $k=-\frac{1}{2}$ and $k=-\frac{4}{3}$ are collapsing levels:
quantum Hamiltonian reduction sends $L_k(\mathfrak{sl}_2)$ to
$\mathbb C$. Consequently, $N_k(\mathfrak{sl}_2)$ is a subalgebra of
the Heisenberg vertex algebra $\mathcal H$ from
Section~\ref{sec:vir-heis-realization}.

Theorem~\ref{thm:admissible-classification} gives the following
classification at these two levels.
\begin{itemize}
\item For $k=-\frac12=-2+\frac32$, we have
$\mathcal H_{3,2}=\{0\}$ and
$\mathcal O^{\mathrm{red}}_{3,2}=\varnothing$.
Hence every irreducible highest weight
$N_{-1/2}(\mathfrak{sl}_2)$--module is of the form
$L_{-1/2}[x,0]$, where $x\in\mathbb C$.
The vertex algebra $N_{-1/2}(\mathfrak{sl}_2)$ is the singlet algebra
$\mathcal M(2)$ of central charge $c=-2$. It is isomorphic to the principal
$W$-algebra $W_{\ell}(\mathfrak{sl}_3,f_{\mathrm{pr}})$ at
$\ell=-\frac{3}{2}$. Thus the above family agrees with the classification
in
\cite{WangW3}.

\item For $k=-\frac43=-2+\frac23$, we again have
$\mathcal H_{2,3}=\{0\}$ and
$\mathcal O^{\mathrm{red}}_{2,3}=\varnothing$.
Therefore every irreducible highest weight
$N_{-4/3}(\mathfrak{sl}_2)$--module is of the form
$L_{-4/3}[x,0]$, where $x\in\mathbb C$.
Since $N_{-4/3}(\mathfrak{sl}_2)\cong\mathcal M(3)$, this recovers the
classification from \cite{A-2003}.
\end{itemize}
The proofs in \cite{A-2003,WangW3} use free-field realizations and
explicit calculations in Zhu's algebras of the singlet algebras. Here
we first prove a general classification result for
$N^k(\mathfrak{sl}_2)$--modules in Theorem~\ref{class-univ}, and the
result then follows from the classification of
$L_k(\mathfrak{sl}_2)$--modules.

\subsection{The supersinglet vertex superalgebra and the orbifold conjecture}

As an application of
Theorem~\ref{thm:admissible-classification},
we verify the orbifold conjecture for  a supersinglet vertex
superalgebra introduced in  \cite{AMSuperTriplet}.

At the admissible level $k=-\frac23$,
the parafermion vertex algebra
$N_{-2/3}(\mathfrak{sl}_2)$
is isomorphic to the $\mathbb Z_2$-orbifold of the supersinglet vertex
superalgebra,
\[
N_{-2/3}(\mathfrak{sl}_2)
\cong\overline{SM(1)}^{\langle\sigma\rangle},
\]
where $\sigma$ denotes the canonical parity automorphism.
The irreducible untwisted $\overline{SM(1)}$-modules are parametrized by
the algebraic curve from \cite{AMSuperTriplet},
while the irreducible $\sigma$-twisted modules are parametrized by the
curve obtained in \cite{AMSuperTripletTwisted}.

By the orbifold theorem of Dong--Mason \cite{DM} (in the vertex superalgebra
setting \cite{ALPY}), every $\sigma$-stable irreducible module decomposes into a
direct sum of two irreducible modules for the orbifold. On the other
hand, if
$M\circ\sigma\not\cong M$,
then $M$ remains irreducible upon restriction to the orbifold.

Every irreducible untwisted $\overline{SM(1)}$-module is
$\sigma$-stable, whereas every irreducible twisted module satisfies
$M\circ\sigma\cong\Pi(M)\not\cong M$.
Hence Theorem~\ref{thm:admissible-classification} immediately yields the
following correspondence.

\medskip

\renewcommand{\arraystretch}{1.25}
\begin{center}
\begin{tabular}{|c|c|}
\hline
$\overline{SM(1)}$-module &
Restriction to
$N_{-2/3}(\mathfrak{sl}_2)$
\\
\hline
irreducible untwisted module &
$L_{-2/3}[x,0]
\oplus
L_{-2/3}\!\left[x,\frac12\right]$
\\
\hline
$\overline{SM(1)}
=
\overline{SM(1)}^{\bar0}
\oplus
\overline{SM(1)}^{\bar1}$ &
$\overline{SM(1)}^{\bar0}
=
L_{-2/3}[0,0],\qquad
\overline{SM(1)}^{\bar1}
=
L_{-2/3}\!\left[-\frac13,\frac53\right]$
\\
\hline
irreducible $\sigma$-twisted module &
$L_{-2/3}\!\left[x,\frac1{16}\right]$
\\
\hline
\end{tabular}
\end{center}

\begin{corollary}
Every irreducible highest weight
$N_{-2/3}(\mathfrak{sl}_2)$-module
is isomorphic to an irreducible summand of the restriction of an
irreducible untwisted or $\sigma$-twisted
$\overline{SM(1)}$-module to
$N_{-2/3}(\mathfrak{sl}_2)$.
\end{corollary}

\begin{remark}
The corollary verifies the orbifold conjecture for the
supersinglet vertex superalgebra.
It follows from Theorem~\ref{thm:admissible-classification}, together
with the known classifications of untwisted and twisted supersinglet
modules \cite{AMSuperTriplet,AMSuperTripletTwisted}.
\end{remark}

\section{Cubic structure of Zhu's algebra
$A(N^k(\mathfrak{sl}_2))$}
\label{sec:cubic-zhu}

Let $R$ be a commutative $\mathbb C$-algebra. A commutative
$R$-algebra $B$ that is finite locally free of rank three as an
$R$-module is called a cubic $R$-algebra \cite{Wood}. The structural
morphism $\operatorname{Spec}B\to\operatorname{Spec}R$ is then finite
and flat of degree three. Such a morphism is also called a triple
cover \cite{Miranda}.

We use the following form of Miranda's construction; see
\cite[Theorem~2.7 and Remark~2.8.2]{Miranda}.

\begin{definition}\label{def:miranda-normal-form}
For $a,b,c,d\in R$, the associated Miranda triple cover algebra is the
commutative algebra
\[
\mathcal C_R(a,b,c,d)
=\frac{R[U,V]}{(F_0,F_1,F_2)},
\]
where
\begin{enumerate}[label=(\arabic*),leftmargin=3em]
\item $F_0=U^2-aU-bV-2(a^2-bd)$;
\item $F_1=UV+dU+aV+(ad-bc)$;
\item $F_2=V^2-cU-dV-2(d^2-ac)$.
\end{enumerate}
\end{definition}

For $k\neq0$, put $A_k=A(N^k(\mathfrak{sl}_2))$. Let $S_k\subset A_k$
be the subalgebra generated by $[L],[W]$ if $k\neq-2$, and by
$[S],[W]$ if $k=-2$. The presentation theorems give
$S_k\cong\mathbb C[w_2,w_3]$ for $k\neq-2$ and
$S_{-2}\cong\mathbb C[y,w_3]$. At the exceptional levels
$k=-17/16,-107/64$, set $(u_k,v_k)=([Z^4],[Z^5])$; at all other
levels, set $(u_k,v_k)=([W^4],[W^5])$.

Theorems~\ref{thm:zhu-presentation} and
\ref{thm:critical-zhu-presentation} show that $A_k$ is a free
$S_k$-module with basis $1,u_k,v_k$. Hence there are unique
$p_i,q_i,r_i\in S_k$ such that
\begin{align*}
u_k^2&=p_0+p_1u_k+p_2v_k,&
u_kv_k&=q_0+q_1u_k+q_2v_k,&
v_k^2&=r_0+r_1u_k+r_2v_k.
\end{align*}
Define
\[
\rho_k=\frac{p_1+q_2}{3},\qquad
\sigma_k=\frac{q_1+r_2}{3},\qquad
\xi_k=u_k-\rho_k,\qquad \zeta_k=v_k-\sigma_k,
\]
and
\[
\alpha_k=\frac{p_1-2q_2}{3},\qquad \beta_k=p_2,\qquad
\gamma_k=r_1,\qquad \delta_k=\frac{r_2-2q_1}{3}.
\]

\begin{theorem}\label{thm:cubic-zhu}
For every $k\neq0$, there is an $S_k$-algebra isomorphism
\[
\mathcal C_{S_k}(\alpha_k,\beta_k,\gamma_k,\delta_k)
\xrightarrow{\ \cong\ }A_k,
\qquad U\longmapsto \xi_k,\quad V\longmapsto \zeta_k.
\]
\end{theorem}

\begin{proof}
The traces of multiplication by $u_k$ and $v_k$ are $p_1+q_2$ and
$q_1+r_2$, respectively. Thus
$\operatorname{tr}_{A_k/S_k}(\xi_k)
=\operatorname{tr}_{A_k/S_k}(\zeta_k)=0$. The definitions above give
\[
\xi_k^2=A+\alpha_k\xi_k+\beta_k\zeta_k,\qquad
\xi_k\zeta_k=B-\delta_k\xi_k-\alpha_k\zeta_k,\qquad
\zeta_k^2=C+\gamma_k\xi_k+\delta_k\zeta_k
\]
for some $A,B,C\in S_k$. The associativity identities
$(\xi_k^2)\zeta_k=\xi_k(\xi_k\zeta_k)$ and
$(\xi_k\zeta_k)\zeta_k=\xi_k(\zeta_k^2)$ give, by comparison of the
coefficients of $\xi_k$ and $\zeta_k$,
\[
A=2(\alpha_k^2-\beta_k\delta_k),\qquad
B=\beta_k\gamma_k-\alpha_k\delta_k,\qquad
C=2(\delta_k^2-\alpha_k\gamma_k).
\]
Hence $U\mapsto \xi_k$ and $V\mapsto \zeta_k$ define an $S_k$-algebra
homomorphism from
$\mathcal C_{S_k}(\alpha_k,\beta_k,\gamma_k,\delta_k)$ to $A_k$.
The defining relations in
Definition~\ref{def:miranda-normal-form} reduce every polynomial to an
$S_k$-linear combination of $1,U,V$. The homomorphism is
surjective because $1,\xi_k,\zeta_k$ form an $S_k$-basis of $A_k$. If
$s_0+s_1U+s_2V$ belongs to its kernel, then
$s_0+s_1\xi_k+s_2\zeta_k=0$, so $s_0=s_1=s_2=0$. It is therefore
injective.
\end{proof}

\appendix
\begingroup\small

\section{Formulas for the relations in Zhu's algebra
$A(N^k(\mathfrak{sl}_2))$}
\label{app-a}

The explicit formulas for the polynomials $R_0,R_1,R_2$ in
Theorem~\ref{thm:zhu-presentation} are given below.

\subsection*{Relations for the generating set $L,W,W^4,W^5$}

\begin{align*}
R_0[k,w_2,w_3,w_4,w_5] &= -8 k^4 (k + 2)^2 (3 k + 4) (4 k - 1) (64 k + 107) (k^2 + k + 1) w_2^2 \\
&\quad + 4 k^4 (k + 2)^3 (3 k + 4) (64 k + 107) (80 k^2 + 30 k + 61) w_2^3 \\
&\quad -112 k^4 (k + 2)^4 (3 k + 4) (6 k - 5) (64 k + 107) w_2^4 \\
&\quad + 2 k (16 k + 17)^2 (k^2 + 3 k + 5) w_3^2 \\
&\quad + k (k + 2) (16 k + 17)^2 (26 k + 83) w_2 w_3^2 \\
&\quad + 2 k^2 (k + 2) (64 k + 107) (8 k^2 + 9 k - 8) w_2 w_4 \\
&\quad - 4 k^2 (k + 2)^2 (36 k + 61) (64 k + 107) w_2^2 w_4 \\
&\quad + 2 (64 k + 107) w_4^2 + (16 k + 17)^2 w_3 w_5
\end{align*}

\begin{align*}
R_1[k, w_2, w_3, w_4, w_5] &= -16 k^3 (k + 2) (2 k + 1) (13 k^3 + 24 k^2 + 7 k + 10) w_2 w_3 \\
&\quad + 4 k^3 (k + 2)^2 (1040 k^3 + 2232 k^2 + 1213 k + 1116) w_2^2 w_3 \\
&\quad - 16 k^3 (k + 2)^3 (674 k^2 + 637 k - 1100) w_2^3 w_3 \\
&\quad + (16 k + 17)(64 k + 107) w_3^3 \\
&\quad + 2 k (68 k^2 + 119 k + 20) w_3 w_4 \\
&\quad - 4 k (k + 2) (358 k + 559) w_2 w_3 w_4 \\
&\quad + 4 k^2 (k + 2) (3 k + 4) (4 k - 1) w_2 w_5 \\
&\quad - 112 k^2 (k + 2)^2 (3 k + 4) w_2^2 w_5 \\
&\quad + 4 w_4 w_5
\end{align*}

\begin{align*}
R_2[k, w_2, w_3,w_4, w_5] &= \frac{1}{2 (17 + 16 k)^2} \Bigl(
64 k^5 (2 + k)^2 (1 + 2 k) (107 + 64 k) (1 + k + k^2)
(10 + k (7 + k (24 + 13 k))) w_2^2 \\
&\quad - 32 k^5 (2 + k)^3 (107 + 64 k)
\bigl(342 + k (1411 + k (3241 + k (3476 + k (2695 + 832 k))))\bigr) w_2^3 \\
&\quad + 16 k^5 (2 + k)^4 (107 + 64 k)
\bigl(2936 + k (14210 + k (19857 + 24630 k + 8936 k^2))\bigr) w_2^4 \\
&\quad + 128 (5 - 6 k) k^5 (2 + k)^5 (107 + 64 k)
\bigl(306 + k (777 + 305 k)\bigr) w_2^5 \\
&\quad + 4 k^2 (2 + k) (17 + 16 k) w_2
\Bigl(
-2 \bigl(2197 + k (6031 + k (10220 + k (7387 + 1840 k)))\bigr) \\
&\qquad + (2 + k) \bigl(29934 + k (91905 + 2 k (35561 + 8848 k))\bigr) w_2
\Bigr) w_3^2 \\
&\quad - 4 k^3 (2 + k) (107 + 64 k) w_2
\Bigl(
80 + 2 k (193 + k (283 + 3 k (103 + 40 k))) \\
&\qquad - (2 + k) (2588 + k (4945 + 4 k (1793 + 772 k))) w_2 \\
&\qquad + 8 (2 + k)^2 (17 + k (1853 + 1108 k)) w_2^2
\Bigr) w_4 \\
&\quad + (17 + 16 k) (107 + 64 k)^2 w_3^2 w_4 \\
&\quad + 2 k (107 + 64 k) (20 + 17 k (7 + 4 k)) w_4^2 \\
&\quad - 16 k (2 + k) (107 + 64 k) (119 + 83 k) w_2 w_4^2 \\
&\quad - 4 k (17 + 16 k)^2 (5 + k (3 + k)) w_3 w_5 \\
&\quad - 2 (17 + 16 k)^2 w_5^2
\Bigr)
\end{align*}

\subsection*{Relations for the generating set $L,W,Z^4,Z^5$}

\begin{align*}
R_0^z[k,w_2,w_3,z_4,z_5]={}&
648k^6(k-4)(k+2)^2(3k+4)^2(7k^2-2k-16)w_2^2\\
&+2592k^6(k+2)^3(3k+4)^2(14k^2-5k-28)w_2^3\\
&-20736k^6(k+2)^4(3k+4)^2w_2^4\\
&+324k^3(2k+3)(27k^3+28k^2-116k-144)w_3^2\\
&+1296k^3(k+2)(2k+3)(4k^2+9k+4)w_2 w_3^2\\
&-72k^3(k+2)(3k+4)(5k^2-5k-16)w_2 z_4\\
&-144k^3(k+2)^2(3k+4)(6k+7)w_2^2 z_4\\
&+2(3k+4)z_4^2-3(2k+3)w_3 z_5.
\end{align*}

\begin{align*}
R_1^z[k,w_2,w_3,z_4,z_5]={}&
1944k^6(k+2)(3k+4)\\
&\quad\cdot(55k^4-46k^3-304k^2+224k+576)w_2 w_3\\
&-7776k^6(k+2)^2(3k+4)^2(6k^2+3k+4)w_2^2 w_3\\
&+62208k^6(k+2)^3(3k+4)^2w_2^3 w_3\\
&-3888k^3(2k+3)^2(3k+4)w_3^3\\
&-108k^3(39k^3+62k^2-92k-144)w_3 z_4\\
&+432k^3(k+2)(32k^2+93k+68)w_2 w_3 z_4\\
&-18k^3(k-4)(k+2)(3k+4)w_2 z_5\\
&-144k^3(k+2)^2(3k+4)w_2^2 z_5+z_4 z_5.
\end{align*}

\begin{align*}
R_2^z[k,w_2,w_3,z_4,z_5]={}&
\frac{23328k^9(k+2)^2(3k+4)^2(7k^2-2k-16)
(55k^4-46k^3-304k^2+224k+576)}{2k+3}w_2^2\\
&-\frac{93312k^9(k+2)^3(3k+4)^2
(98k^5+307k^4+42k^3-832k^2-736k+64)}{2k+3}w_2^3\\
&+\frac{746496k^9(k+2)^4(3k+4)^2
(28k^4+83k^3-12k^2-180k-112)}{2k+3}w_2^4\\
&-\frac{11943936k^9(k+1)(k+2)^6(3k+4)^2}{2k+3}w_2^5\\
&-93312k^6(k+2)(9k^5+67k^4+197k^3+218k^2-96)w_2 w_3^2\\
&+373248k^6(k+2)^3(8k^3+29k^2+36k+16)w_2^2 w_3^2\\
&-\frac{2592k^6(k+2)(3k+4)
(219k^5+219k^4-1244k^3-1180k^2+2192k+2304)}{2k+3}w_2 z_4\\
&+\frac{5184k^6(k+2)^2(3k+4)
(238k^4+747k^3+222k^2-1244k-1040)}{2k+3}w_2^2 z_4\\
&-\frac{82944k^6(k+2)^3(3k+4)
(6k^3+33k^2+57k+32)}{2k+3}w_2^3 z_4\\
&+2592k^3(2k+3)(3k+4)^2w_3^2 z_4\\
&+\frac{72k^3(3k+4)(39k^3+62k^2-92k-144)}{2k+3}z_4^2\\
&-\frac{1152k^3(k+2)(3k+4)(7k^2+21k+16)}{2k+3}w_2 z_4^2\\
&+108k^3(27k^3+28k^2-116k-144)w_3 z_5-z_5^2.
\end{align*}

\section{Polynomial reductions for Zhu's algebra
$A(N^k(\mathfrak{sl}_2))$}
\label{app-b}
This appendix gives the reductions used in
Lemma~\ref{lem:character-value-independence}.  Let
$S=\mathbb C[w_2,w_3]$ and put
\[
C_{w_2,w_3}(t)=8t^3-12kt^2+
\bigl(4k^2+12k(k+2)w_2\bigr)t
+w_3-6k^2(k+2)w_2.
\]
We work in the free $S$-module
$S[t]/(C_{w_2,w_3}(t))=S\oplus St\oplus St^2$ and substitute
$x=t$ and $y=w_2-t+t^2/k$.

\subsection*{Character values on $[W^4]$ and $[W^5]$}

For $k\neq0,-2,-17/16,-107/64$, reduction modulo
$C_{w_2,w_3}(t)$ in $S[t]$ gives
\begin{align*}
f_4\left(k,t,w_2-t+\frac{t^2}{k}\right)
&\equiv a_0+a_1t+a_2t^2,\\
f_5\left(k,t,w_2-t+\frac{t^2}{k}\right)
&\equiv b_0+b_1t+b_2t^2,
\end{align*}
where $a_0,b_0\in S$ and
\begin{align*}
a_1&=(16k+17)\bigl(4w_2k^2(k+2)-w_3\bigr),\\
a_2&=-4w_2k(k+2)(16k+17),\\
b_1&=2k(64k+107)
\bigl(w_3-w_2k^2(k+2)+8w_2^2k(k+2)^2\bigr),\\
b_2&=-2(64k+107)w_3.
\end{align*}
Consequently,
\[
\begin{aligned}
a_1b_2-a_2b_1
={}&2(16k+17)(64k+107)\\
&\quad\cdot
\bigl(w_3^2-4w_2^2k^4(k+2)^2
+32w_2^3k^3(k+2)^3\bigr).
\end{aligned}
\]

\subsection*{Character values on $[Z^4]$ and $[Z^5]$}

For the character values of the generators $Z^4$ and $Z^5$, reduction
modulo $C_{w_2,w_3}(t)$ in $S[t]$ gives
\begin{align*}
g_4\left(k,t,w_2-t+\frac{t^2}{k}\right)
&\equiv c_0+c_1t+c_2t^2,\\
g_5\left(k,t,w_2-t+\frac{t^2}{k}\right)
&\equiv d_0+d_1t+d_2t^2,
\end{align*}
where $c_0,d_0\in S$ and
\begin{align*}
c_1&=36k(2k+3)\bigl(4w_2k^2(k+2)-w_3\bigr),\\
c_2&=-144w_2k^2(k+2)(2k+3),\\
d_1&=864k^3(2k+3)(3k+4)\\
&\qquad\cdot
\bigl(w_2k^2(k+2)-8w_2^2k(k+2)^2-w_3\bigr),\\
d_2&=864k^2(2k+3)(3k+4)w_3.
\end{align*}
It follows that
\[
\begin{aligned}
c_1d_2-c_2d_1
={}&-31104k^3(2k+3)^2(3k+4)\\
&\quad\cdot
\bigl(w_3^2-4w_2^2k^4(k+2)^2
+32w_2^3k^3(k+2)^3\bigr).
\end{aligned}
\]
At $k=-17/16$ and $k=-107/64$, the scalar factor in this determinant
is nonzero.

\subsection*{Character values at the critical level}

Let
\[
\mathcal S=\mathbb C[y,w_3],
\qquad
C^{\mathrm{crit}}_{y,w_3}(t)
=8t^3-(12y+8)t+w_3.
\]
We work in the free $\mathcal S$-module
$\mathcal S[t]/(C^{\mathrm{crit}}_{y,w_3}(t))
=\mathcal S\oplus\mathcal S t\oplus\mathcal S t^2$.
For the critical character values
\eqref{eq:critical-w4} and \eqref{eq:critical-w5}, division by
$C^{\mathrm{crit}}_{y,w_3}(t)$ gives
\begin{align*}
w_4(t,y)
&\equiv34y^2+24y+15w_3t-60yt^2,\\
w_5(t,y)
&\equiv(73y+12)w_3-168y(2y+1)t+42w_3t^2.
\end{align*}
Consequently, the determinant of the coefficients of $t$ and $t^2$ is
\[
630\bigl(w_3^2-16y^2(2y+1)\bigr)
\in\mathcal S\setminus\{0\}.
\]

\section*{Acknowledgements}

The authors acknowledge the use of OpenAI's ChatGPT as an auxiliary
tool during the preparation of this manuscript, including for
mathematical discussion and editorial assistance. All mathematical
content was independently verified by the authors, who take full
responsibility for the final manuscript.

D.A. is partially supported by the Croatian Science Foundation under
project IP-2022-10-9006 and by the project Implementation of
cutting-edge research and its application as part of the Scientific Center
of Excellence for Quantum and Complex Systems, and Representations of Lie
Algebras, grant no.~PK.1.1.10.0004, co-financed by the European Union
through the European Regional Development Fund, Competitiveness and
Cohesion Programme 2021--2027.
Q.W. is supported by the National Natural Science Foundation of China,
grant No.~12571033.

\endgroup


\begin{thebibliography}{99}
\small

\bibitem{A-2003}
D. Adamovi\'c,
\emph{Classification of irreducible modules of certain subalgebras of free boson vertex algebra},
J. Algebra \textbf{270} (2003), 115--132.

\bibitem{A-2005}
D. Adamovi\'c,
\emph{A construction of admissible $A_1^{(1)}$-modules of level $-\frac{4}{3}$},
J. Pure Appl. Algebra \textbf{196} (2005), 119--134.

\bibitem{A2007}
D. Adamovi\'c,
\emph{Lie superalgebras and irreducibility of
$A_1^{(1)}$-modules at the critical level},
Comm. Math. Phys. \textbf{270} (2007), 141--161.

\bibitem{AM95}
D. Adamovi\'c and A. Milas,
\emph{Vertex operator algebras associated to modular invariant
representations of $A^{(1)}_1$},
Math. Res. Lett. \textbf{2} (1995), 563--575.


\bibitem{AMSuperTriplet}
D. Adamovi\'c and A. Milas,
\emph{The $N=1$ triplet vertex operator superalgebras},
Comm. Math. Phys. \textbf{288} (2009), no. 1, 225--270.

\bibitem{AMSuperTripletTwisted}
D. Adamovi\'c and A. Milas,
\emph{The $N=1$ triplet vertex operator superalgebras: twisted sector},
SIGMA Symmetry Integrability Geom. Methods Appl. \textbf{4} (2008), Paper 087, 24 pp.


\bibitem{AKMPP17}
D. Adamovi\'c, V. G. Kac, P. Moseneder Frajria, P. Papi and O. Per\v se,
\emph{Conformal embeddings of affine vertex algebras in minimal
$W$-algebras II: decompositions},
Jpn. J. Math. \textbf{12} (2017), 261--315.

\bibitem{ArakawaCriticalW}
T. Arakawa,
\emph{$W$-algebras at the critical level},
Contemp. Math. \textbf{565} (2012), 1--14.

\bibitem{AtiyahMacdonald}
M. F. Atiyah and I. G. Macdonald,
\emph{Introduction to commutative algebra},
Addison--Wesley, Reading, MA, 1969.

\bibitem{A2019}
D. Adamovi\'c,
\emph{Realizations of simple affine vertex algebras and their modules: the cases
$\widehat{\mathfrak{sl}}_2$ and $\widehat{\mathfrak{osp}}(1,2)$},
Comm. Math. Phys. \textbf{366} (2019), no. 3, 1025--1067.


\bibitem{ACR}
J. Auger, T. Creutzig and D. Ridout,
\emph{Modularity of logarithmic parafermion vertex algebras},
Lett. Math. Phys. \textbf{108} (2018), no. 11, 2543--2587.

\bibitem{AMW}
D. Adamovi\'c, A. Milas and Q. Wang,
\emph{On parafermion vertex algebras of $\mathfrak{sl}_2$ and $\mathfrak{sl}_3$ at level $-3/2$},
Commun. Contemp. Math. \textbf{24} (2022), no. 1, Paper No. 2050086, 23 pp.


\bibitem{ALPY}
D. Adamovi\'c, C. H. Lam, V. Pedi\'c Tomi\'c and N. Yu,
\emph{On irreducibility of modules of Whittaker type: twisted modules and
nonabelian orbifolds},
J. Pure Appl. Algebra \textbf{229} (2025), Paper No.~107840.


\bibitem{AKR24}
D. Adamovi\'c, K. Kawasetsu and D. Ridout,
\emph{Weight module classifications for Bershadsky--Polyakov algebras},
Commun. Contemp. Math. \textbf{26} (2024), no.~10,
Paper No.~2350063.

\bibitem{ALY}
T. Arakawa, C. H. Lam and H. Yamada,
\emph{Zhu's algebra, $C_2$-algebra and $C_2$-cofiniteness of parafermion vertex operator algebras},
Adv. Math. \textbf{264} (2014), 261--295.


\bibitem{BEHHH}
R. Blumenhagen, W. Eholzer, A. Honecker, K. Hornfeck and R. H\"ubel,
\emph{Coset realization of unifying $W$-algebras},
Internat. J. Modern Phys. A \textbf{10} (1995), 2367--2430.

\bibitem{CreutzigTensorSL2}
T. Creutzig,
\emph{Tensor categories of weight modules of
$\widehat{\mathfrak{sl}}_2$ at admissible level},
J. Lond. Math. Soc. \textbf{110} (2024), no.~6,
Paper No.~e70037.



\bibitem{CKLR}
T. Creutzig, S. Kanade, A. R. Linshaw and D. Ridout,
\emph{Schur--Weyl duality for Heisenberg cosets},
Transform. Groups \textbf{24} (2019), no.~2, 301--354.


\bibitem{CMYRibbonSL2}
T. Creutzig, R. McRae and J. Yang,
\emph{Ribbon categories of weight modules for affine
$\mathfrak{sl}_2$ at admissible levels},
Comm. Math. Phys. \textbf{407} (2026), no.~6,
Paper No.~123.

\bibitem{DL}
C. Dong and J. Lepowsky,
\emph{Generalized vertex algebras and relative vertex operators},
Progress in Mathematics, vol.~112,
Birkh\"auser Boston, Boston, MA, 1993.


\bibitem{DLY}
C. Dong, C. H. Lam and H. Yamada,
\emph{$W$-algebras related to parafermion algebras},
J. Algebra \textbf{322} (2009), 2366--2403.

\bibitem{DLWY}
C. Dong, C. H. Lam, Q. Wang and H. Yamada,
\emph{The structure of parafermion vertex operator algebras},
J. Algebra \textbf{323} (2010), no.~2, 371--381.


\bibitem{DW1} C. Dong and Q. Wang, \emph{The structure of parafermion vertex operator algebras: general case}, Comm. Math. Phys.
 \textbf {299} (2010), 783--792.


\bibitem{DM}
C. Dong and G. Mason,
\emph{On quantum Galois theory},
Duke Math. J. \textbf{86} (1997), no.~2, 305--321.


\bibitem{DongRen}
C. Dong and L. Ren,
\emph{Representations of the parafermion vertex operator algebras},
Adv. Math. \textbf{315} (2017), 88--101.


 \bibitem{FZ}
I. B. Frenkel and Y. Zhu,
\emph{Vertex operator algebras associated to representations of affine and Virasoro algebras},
Duke Math. J. \textbf{66} (1992), 123--168.

\bibitem{GarbagnatiPenegini}
A. Garbagnati and M. Penegini,
\emph{Triple covers of K3 surfaces},
Nagoya Math. J. \textbf{248} (2022), 939--979.

\bibitem{GK}
M. Gorelik and V. G. Kac,
\emph{On simplicity of vacuum modules},
Adv. Math. \textbf{211} (2007), no.~2, 621--677.

\bibitem{HuangC1}
Y.-Z. Huang,
\emph{$C_1$-cofiniteness and vertex tensor categories},
arXiv:2509.20737.

\bibitem{LinshawWinfty}
A. R. Linshaw,
\emph{Universal two-parameter $\mathcal W_\infty$-algebra and vertex algebras
of type $\mathcal W(2,3,\ldots,N)$},
Compos. Math. \textbf{157} (2021), no.~1, 12--82.

\bibitem{Miranda}
R. Miranda,
\emph{Triple covers in algebraic geometry},
Amer. J. Math. \textbf{107} (1985), no.~5, 1123--1158.


\bibitem{NORW}
H. Nakano, F. Orosz Hunziker, A. Ros Camacho and S. Wood,
\emph{Fusion rules and rigidity for weight modules over the simple
admissible affine $\mathfrak{sl}_2$ and $\mathcal N=2$
superconformal vertex operator superalgebras},
Adv. Math. \textbf{502} (2026), Part A,
Paper No.~111124.

\bibitem{WangW3}
W. Wang,
\emph{Classification of irreducible modules of $W_3$ algebra with
$c=-2$},
Comm. Math. Phys. \textbf{195} (1998), 113--128.



\bibitem{Wood}
M. M. Wood,
\emph{Rings and ideals parameterized by binary $n$-ic forms},
J. Lond. Math. Soc. (2) \textbf{83} (2011), no.~1, 208--231.

\bibitem{Zhu}
Y. Zhu,
\emph{Modular invariance of characters of vertex operator algebras},
J. Amer. Math. Soc. \textbf{9} (1996), 237--302.

\end{thebibliography}
\end{document}